\documentclass[12pt,reqno]{amsart}
\usepackage{amsmath,amsthm,amsfonts,amssymb,times}
\usepackage{verbatim}
\usepackage{url}
\usepackage{amssymb,latexsym}
\usepackage{mathtools}
\usepackage{graphicx}
\usepackage{url}	

\usepackage[utf8]{inputenc}
\usepackage[T1]{fontenc}
\usepackage{amsfonts}
\usepackage[greek,english]{babel}

\setbox0=\hbox{$+$}
\newdimen\plusheight
\plusheight=\ht0
\def\+{\;\lower\plusheight\hbox{$+$}\;}

\setbox0=\hbox{$-$}
\newdimen\minusheight
\minusheight=\ht0
\def\-{\;\lower\minusheight\hbox{$-$}\;}

\setbox0=\hbox{$\cdots$}
\newdimen\cdotsheight
\cdotsheight=\plusheight%\ht0
\def\cds{\lower\cdotsheight\hbox{$\cdots$}}

\makeatletter
\def\leqalignno#1{\displ@y \tabskip\z@ plus\@ne fil
  \halign to\displaywidth{\hfil$\@lign\displaystyle{##}$\tabskip\z@skip
    &$\@lign\displaystyle{{}##}$\hfil\tabskip\z@ plus\@ne fil
    &\kern-\displaywidth\rlap{$\@lign\hbox{\rm##}$}\tabskip\displaywidth\crcr
    #1\crcr}}
\makeatother

\newcommand{\df}{\dfrac}

\renewcommand{\Re}{\text{Re}}

\renewcommand{\(}{\left\(}
\renewcommand{\)}{\right\)}
\renewcommand{\[}{\left\[}
\renewcommand{\]}{\right\]}
\numberwithin{equation}{section}
\theoremstyle{plain}
\newtheorem{theorem}{Theorem}[section]
\newtheorem{lemma}[theorem]{Lemma}
\newtheorem{corollary}[theorem]{Corollary}

\newtheorem{entry}[theorem]{Entry}

\makeatletter
\newcommand{\leqnomode}{\tagsleft@true\let\veqno\@@leqno}
\newcommand{\reqnomode}{\tagsleft@false\let\veqno\@@eqno}
\makeatother

\allowdisplaybreaks

\begin{document}
\title[Cubic and quintic analogues]{Cubic and quintic analogues of\\ Ramanujan's septic theta function identity}
\author{Bruce C.~Berndt and \"{O}rs Reb\'{a}k}
\address{Department of Mathematics, University of Illinois, 1409 West Green Street, Urbana, IL 61801, USA}
\email{berndt@illinois.edu}
\address{Department of Mathematics and Statistics, University of Troms\o{} -- The Arctic University of Norway, 9037 Troms\o{}, Norway}
\email{ors.rebak@uit.no}

\date{\today}

\begin{abstract}
On page~206 in his lost notebook, Ramanujan recorded an incomplete septic theta function identity. Motivated by the completion of this identity by the second author, we offer cubic and quintic analogues. Using the theory generated by these two analogues and Ramanujan's class invariants, we provide many evaluations for Ramanujan's most prominent theta function, $\varphi(q)$ in his notation.
\end{abstract}

\maketitle

\section{Introduction}
In his lost notebook \cite[p.~206]{RamanujanLost}, Ramanujan recorded an incomplete formula, with three missing terms and a misprint, for
$\varphi(e^{-7\pi\sqrt{7}})$,
where, in the notation of Ramanujan,
\begin{equation}\label{def:phi}
\varphi(q):=\sum_{n=-\infty}^{\infty}q^{n^2},\quad |q|<1.
\end{equation}
In their second book on Ramanujan's lost notebook \cite[pp.~180--194]{AndrewsBerndtII}, G.~E.~Andrews and the present first author briefly discussed Ramanujan's deficient entry, but did not supply the missing terms. The second author of the present paper derived the missing terms in \cite{Rebak}, and thereby completed a remarkable entry of Ramanujan from his lost notebook.

The goal of this paper is to establish and prove cubic and quintic analogues of Ramanujan's now completed entry. The foundations for these two theories were set by Ramanujan in his notebooks \cite{RamanujanEarlierI}, \cite{RamanujanEarlierII}. As corollaries, a multitude of explicit values of $\varphi(e^{-\pi\sqrt{n}})$, where $n$ is a positive rational number, are established. Most of these evaluations are new.

The cubic, quintic, and septic theories are evidently special cases of a grand theory that Ramanujan envisioned at the end of Section~12 of Chapter~20 in his second notebook \cite[p.~247]{RamanujanEarlierII}, \cite[p.~400]{BerndtIII}.

In the last section of our paper, we offer several evaluations for Ramanujan's cubic theta function, also known as the Borweins' cubic theta function \cite{BorweinBrothersCubic},
\begin{equation}\label{def:a}
a(q):=\sum_{m,n=-\infty}^{\infty}q^{m^2+mn+n^2},\quad |q|<1.
\end{equation}
As above, the majority of these determinations are new.

In order to state the aforementioned theorem of Ramanujan, we need to offer some notation.
Let
\begin{equation*}
(a;q)_{\infty} := \prod_{k=0}^{\infty}(1-aq^k), \qquad |q|<1.
\end{equation*}
After Ramanujan \cite[p.~37]{BerndtIII}, set
\begin{equation}\label{def:chi}
\chi(q):=(-q;q^2)_{\infty}.
\end{equation}
When $q=e^{-\pi\sqrt{n}}$, for a positive rational $n$, the class invariant $G_n$ is defined by \cite[pp.~21, 183]{BerndtV}
\begin{equation}\label{def:G}
G_n:=2^{-1/4}q^{-1/24}\chi(q).
\end{equation}
Ramanujan discussed properties of class invariants in his paper \cite{RamanujanModularPi}, \cite[pp.~23--39]{RamanujanCollected}. In particular, we  need the property
\begin{equation}\label{eq:G}
G_n=G_{1/n}.
\end{equation}

Ramanujan's general theta function $f(a,b)$ is defined by \cite[p.~197]{RamanujanEarlierII}, \cite[p.~34]{BerndtIII}
\begin{equation}\label{eq:Jacobi-triple-product}
f(a,b):=\sum_{n=-\infty}^{\infty}a^{n(n+1)/2}b^{n(n-1)/2}= (-a;ab)_{\infty}(-b;ab)_{\infty}(ab;ab)_{\infty}, \qquad |ab| <1,
\end{equation}
where the latter representation is the Jacobi triple product identity \cite[pp.~176--183]{Jacobi}, \cite[p.~197]{RamanujanEarlierII}, \cite[p.~35, Entry~19]{BerndtIII}. Comparing \eqref{eq:Jacobi-triple-product} with \eqref{def:phi}, note that $f(q,q)=\varphi(q)$.

We next offer the classical theta transformation formula \cite[p.~199]{RamanujanEarlierII}, \cite[p.~43, Entry~27(i)]{BerndtIII}. If $\Re(\alpha^2), \Re(\beta^2) > 0$ and $\alpha\beta=\pi$, then
\begin{equation*}
\sqrt{\alpha}\varphi(e^{-\alpha^2})=\sqrt{\beta}\varphi(e^{-\beta^2}).
\end{equation*}
In the special case, if $n$ is a positive rational number and $\alpha^2 = \pi/\sqrt{n}$, then
\begin{equation}\label{eq:transform}
\varphi(e^{-\pi/\sqrt{n}})=n^{1/4}\varphi(e^{-\pi\sqrt{n}}).
\end{equation}

We are now ready to offer Ramanujan's incomplete septic identity from his lost notebook \cite[p.~206]{RamanujanLost} as he wrote it (but with a misprint corrected).

\begin{entry}[p.~206]\label{entry:Ramanujan}
\leqnomode
Let
\begin{equation}\label{7:1+u+v+w}
\frac{\varphi(q^{1/7})}{\varphi(q^7)} = 1 + u + v + w.\tag{i}
\end{equation}
Then, for $u$, $v$, and $w$ given below,
\begin{equation}\label{def:7:p}
p := uvw = \frac{8q^2 (-q; q^2)_\infty}{(-q^7; q^{14})^7_\infty}\tag{ii}
\end{equation}
and
\begin{align}\label{7:phi-to-p}
\frac{\varphi^8(q)}{\varphi^8(q^7)} - (2 + 5p)\frac{\varphi^4(q)}{\varphi^4(q^7)} + (1 - p)^3 = 0.\tag{iii}
\end{align}
Furthermore,
\begin{equation}\label{7:u,v,w}
u = \bigg(\frac{\alpha^2 p}{\beta}\bigg)^{1/7}, \quad
v = \bigg(\frac{\beta^2 p}{\gamma}\bigg)^{1/7}, \text{\quad and \quad}
w = \bigg(\frac{\gamma^2 p}{\alpha}\bigg)^{1/7},\tag{iv}
\end{equation}
where $\alpha, \beta,$ and $\gamma$ are the roots of the cubic equation
\begin{equation}\label{7:r}
r(\xi) := \xi^3 + 2\xi^2\bigg(1 + 3p - \frac{\varphi^4(q)}{\varphi^4(q^7)}\bigg) + \xi p^2(p+4) - p^4 = 0.\tag{v}
\end{equation}
For example,
\begin{equation}\label{enigmatic}
\varphi(e^{-7\pi\sqrt{7}}) = 7^{-3/4}\varphi(e^{-\pi\sqrt{7}})\Big\{1 + (\quad)^{2/7} + (\quad)^{2/7} + (\quad)^{2/7}\Big\}.\tag{vi}
\end{equation}
\reqnomode
\end{entry}
Note that $u, v,$ and $w$ depend on the order of the roots $\alpha, \beta,$ and $\gamma$. Part~\eqref{7:1+u+v+w} is recorded in Ramanujan's second notebook \cite[p.~239]{RamanujanEarlierII}, \cite[p.~303]{BerndtIII} as well, but in the form
\begin{equation*}
\varphi(q^{1/7}) - \varphi(q^7) = 2q^{1/7}f(q^5, q^9) + 2q^{4/7}f(q^3, q^{11}) + 2q^{9/7}f(q, q^{13}),
\end{equation*}
from which we can deduce the definitions \cite{Son}, \cite[p.~181]{AndrewsBerndtII}, \cite[p.~198]{Son2}, \cite{Rebak}
\begin{equation}\label{def:u,v,w}
u:=2q^{1/7}\frac{f(q^5, q^9)}{\varphi(q^7)}, \qquad
v:=2q^{4/7}\frac{f(q^3, q^{11})}{\varphi(q^7)}, \qquad
w:=2q^{9/7}\frac{f(q, q^{13})}{\varphi(q^7)}.
\end{equation}

Parts \eqref{7:1+u+v+w}--\eqref{7:r} were proved by Seung Hwan Son \cite{Son}, \cite[pp.~198--200]{Son2}, \cite[pp.~180--194]{AndrewsBerndtII}.
The second author completed \eqref{enigmatic}, and thereby Ramanujan's Entry~\ref{entry:Ramanujan}, by establishing the following representation \cite[Theorem~4.1]{Rebak}.

\begin{theorem}\label{thm:enigmatic} We have
\begin{equation*}
\varphi(e^{-7\pi\sqrt{7}}) = 7^{-3/4}\varphi(e^{-\pi\sqrt{7}})\Bigg\{1 + \bigg(\frac{\cos\frac{\pi}{7}}{2\cos^2\frac{2\pi}{7}}\bigg)^{2/7} + \bigg(\frac{\cos\frac{2\pi}{7}}{2\cos^2\frac{3\pi}{7}}\bigg)^{2/7} + \bigg(\frac{\cos\frac{3\pi}{7}}{2\cos^2\frac{\pi}{7}}\bigg)^{2/7}\Bigg\}.
\end{equation*}
\end{theorem}

As in this first septic example in Theorem~\ref{thm:enigmatic}, the primary cubic and quintic analogous examples can be eloquently expressed in terms of trigonometric functions, as they are given in Corollaries~\ref{cor:3sqrt3} and~\ref{cor:5sqrt5}, respectively.

Finding a specific value of $\varphi(q)$, in particular, of $\varphi(e^{-\pi\sqrt{n}})$ for a positive rational number $n$, is equivalent to determining a specific value of a complete elliptic integral of the first kind, and also to determining the value of a certain ordinary hypergeometric series. These two approaches focus on elliptic integrals and their relations to the theory of positive-definite binary quadratic forms. Theorem~\ref{thm:thetagamma} below provides evaluations that are expressed in terms of gamma functions. Along these lines, a famous result of A.~Selberg and S.~Chowla \cite{SelbergChowla} provides a path toward the evaluation of $\varphi(e^{-\pi\sqrt{n}})$ for each positive rational number $n$.

\begin{theorem}\label{thm:thetagamma} We have
\begin{align}
\varphi(e^{-\pi\sqrt{3}}) &= \frac{3^{1/8} \Gamma^{3/2}\big(\tfrac{1}{3}\big)}{2^{2/3} \pi},\label{eq:e3}\\
\varphi(e^{-\pi\sqrt{5}}) &= (\sqrt{5} + 2)^{1/8}\left(\frac{\Gamma\big(\tfrac{1}{20}\big)\Gamma\big(\tfrac{3}{20}\big)\Gamma\big(\tfrac{7}{20}\big)
\Gamma\big(\tfrac{9}{20}\big)}{40\pi^3}\right)^{1/4},\\
\varphi(e^{-\pi\sqrt{7}}) &= \frac{\big\{\Gamma\big(\tfrac{1}{7}\big)\Gamma\big(\tfrac{2}{7}\big)\Gamma\big(\tfrac{4}{7}\big)\big\}^{1/2}}{\sqrt{2}\cdot7^{1/8}\pi},\\
\varphi(e^{-\pi\sqrt{11}}) &= (2 + (3\sqrt{33} + 17)^{1/3} - (3\sqrt{33} - 17)^{1/3})\notag\\
&\qquad\times\left(\frac{\Gamma\big(\tfrac{1}{11}\big)\Gamma\big(\tfrac{3}{11}\big)\Gamma\big(\tfrac{4}{11}\big)
\Gamma\big(\tfrac{5}{11}\big)\Gamma\big(\tfrac{9}{11}\big)}{72 \cdot 11^{1/4}\pi^3}\right)^{1/2},\\
\varphi(e^{-\pi\sqrt{13}}) &= (18 + 5\sqrt{13})^{1/8}\notag\\
\MoveEqLeft[2.5]\times\left(\frac{\Gamma\big(\frac{1}{52}\big)\Gamma\big(\frac{7}{52}\big)\Gamma\big(\frac{9}{52}\big)
\Gamma\big(\frac{11}{52}\big)\Gamma\big(\frac{15}{52}\big)\Gamma\big(\frac{17}{52}\big)
\Gamma\big(\frac{19}{52}\big)\Gamma\big(\frac{25}{52}\big)\Gamma\big(\frac{29}{52}\big)
\Gamma\big(\frac{31}{52}\big)\Gamma\big(\frac{47}{52}\big)\Gamma\big(\frac{49}{52}\big)}{1664 \pi^7}\right)^{1/4}\label{eq:e13},
\intertext{and}
\varphi(e^{-\pi\sqrt{17}}) &= 2^{-7/4}(17)^{-1/4}\pi^{-1/4}(\sqrt{17} - 4)^{1/16}\notag\\
&\qquad\times\bigg(1 + \sqrt{17} + \sqrt{2 + 2\sqrt{17}}\bigg)^{3/4}\Bigg\{\prod_{m = 1}^{68} \Gamma\bigg(\frac{m}{68}\bigg)^{\big(\frac{-68}{m}\big)}\Bigg\}^{1/16},\label{eq:e17}
\end{align}
where $\big(\tfrac{n}{m}\big)$ denotes the Kronecker symbol.
\end{theorem}

The first five evaluations \eqref{eq:e3}--\eqref{eq:e13} are given by J.~M.~Borwein and I.~J.~Zucker~\cite{Zucker}, \cite{BorweinZucker}, \cite[p.~298, Table~9.1]{BorweinBrothersPiAGM}, who evaluated $K(\sqrt{n})$, $1\leq n\leq 16$. Zucker~\cite{Zucker} used the theory of positive-definite quadratic forms, Dirichlet $L$-series, a formula of Dirichlet relating values of $L$-functions with values of the Dedekind eta-function $\eta(\tau)$ and related functions, and ideas from the classical paper of Selberg and Chowla \cite{SelbergChowla}.

In their paper \cite{MuzaffarWilliams}, H.~Muzaffar and K.~S.~Williams first establish a general theorem through the theory of positive-definite, primitive, integral, binary quadratic forms \cite[\linebreak pp.~1643--1645, Section~4]{MuzaffarWilliams}. They then use their general theorem to work out the special case when the discriminant equals $-68$ \cite[pp.~1645--1659, Section~5]{MuzaffarWilliams}. Their evaluation is equivalent to \eqref{eq:e17}.

We also provide the following value.

\begin{theorem}\label{thm:e37} We have
\begin{equation*}
\varphi(e^{-\pi\sqrt{37}}) = 2^{-1/4}(37)^{-1/4}\pi^{-1/4}(6 + \sqrt{37})^{3/8}\Bigg\{\prod_{m = 1}^{148} \Gamma\bigg(\frac{m}{148}\bigg)^{\big(\frac{-148}{m}\big)}\Bigg\}^{1/8}.
\end{equation*}
\end{theorem}

\begin{proof}
The Dedekind eta function $\eta(\tau)$ is defined by~\cite[p.~256]{Cox}, \cite[p.~323]{BerndtV}
\begin{equation}\label{def:eta}
\eta(\tau) := q^{1/24} \prod_{m = 1}^{\infty} (1 - q^m) = q^{1/24} (q; q)_\infty, \quad q = e^{2\pi i \tau}, \quad \operatorname{Im} \tau > 0.
\end{equation}
It follows that \cite[pp.~259--260]{Cox}
\begin{equation}\label{eq:eta-transf}
\eta(\tau + 1) = e^{\pi i/12}\eta(\tau).
\end{equation}
The third Jacobi theta function $\theta_3(z, q)$ is defined by~\cite[pp.~463--464]{WhittakerWatson}, \cite[p.~3]{BerndtIII}
\begin{equation}\label{def:theta_3-q}
\theta_3(z, q) := \sum_{m = -\infty}^{\infty} q^{m^2}e^{2m i z}, \qquad z \in \mathbb{C}, \quad |q| < 1.
\end{equation}
From \eqref{def:phi}, \eqref{def:theta_3-q}, \eqref{def:eta}, and the Jacobi triple product identity \eqref{eq:Jacobi-triple-product}, as shown in \cite[p.~46, Theorem~12]{Knopp}, for $\operatorname{Im} \tau > 0$,
\begin{equation}\label{eq:phi-eta}
\varphi(e^{\pi i \tau}) = \theta_3(0, e^{\pi i \tau}) = \frac{\eta^2\big(\frac{\tau + 1}{2}\big)}{\eta(\tau + 1)}.
\end{equation}

For a positive rational number $m$, let $\tau = \sqrt{-m}$. Then, by \eqref{def:G} and \eqref{def:eta}, we obtain \cite[p.~220, (4.19)]{BerndtV}
\begin{equation}\label{eq:eta-G}
\frac{\eta\big(\frac{\tau + 1}{2}\big)}{\eta(\tau)} = 2^{1/4}G_m.
\end{equation}
Thus, from \eqref{eq:phi-eta}, \eqref{eq:eta-transf}, and \eqref{eq:eta-G}, with $\tau = \sqrt{-m}$,
\begin{equation}\label{eq:phi-eta-2}
\Bigg|\frac{\varphi^8(e^{\pi i \tau})}{\eta^4\big(\frac{\tau + 1}{2}\big)\eta^4(\tau + 1)}\Bigg|^2 = \Bigg|\frac{\eta^{24}\big(\frac{\tau + 1}{2}\big)}{\eta^{24}(\tau + 1)}\Bigg| = \Bigg|\frac{\eta^{24}\big(\frac{\tau + 1}{2}\big)}{\eta^{24}(\tau)}\Bigg| = 2^{6}G_m^{24}.
\end{equation}

Let $m = 37$. Rearranging \eqref{eq:phi-eta-2}, noting that $\varphi(e^{-\pi \sqrt{37}})$ is real, and using  \eqref{eq:eta-transf}, we find that
\begin{equation}\label{eq:e37}
\varphi^{16}(e^{-\pi \sqrt{37}}) = \bigg|\eta^4\bigg(\frac{\sqrt{-37}}{2} + \frac{1}{2}\bigg)\eta^4(\sqrt{-37})\bigg|^2 \cdot 2^{6}G_{37}^{24}.
\end{equation}
Now, from Ramanujan's \cite[p.~191]{BerndtV} or Weber's list \cite[p.~722]{Weber}, we know that
\begin{equation}\label{G37}
G_{37}^4 = 6 + \sqrt{37}.
\end{equation}
Furthermore, we know that \cite[p.~85, Table~I]{Dickson} the two positive, reduced, primitive forms $[a, b, c]$ of the fundamental discriminant $-148$ are $[1, 0, 37]$ and $[2, 2, 19]$. Thus, by using the Selberg--Chowla formula \cite[p.~110, (2)]{SelbergChowla}, as it is given in \cite[(1.5)]{HuardKaplanWilliams}, we have
\begin{equation}\label{eq:SelbergChowla}
\bigg|\eta^4\bigg(\frac{\sqrt{-37}}{2} + \frac{1}{2}\bigg)\eta^4(\sqrt{-37})\bigg| = 2^{-5} (37)^{-2} \pi^{-2} \prod_{m = 1}^{148} \Gamma\bigg(\frac{m}{148}\bigg)^{\big(\frac{-148}{m}\big)}.
\end{equation}
The proof is completed by substituting \eqref{G37} and \eqref{eq:SelbergChowla} into \eqref{eq:e37} and taking the $16$th root.
\end{proof}

We emphasize that all of our evaluations of theta function quotients in this paper are given by algebraic numbers. Most of our results can be expressed in terms of gamma functions by using the values in Theorems~\ref{thm:thetagamma} and~\ref{thm:e37}. Perhaps the approach through Ramanujan's ideas is simpler than approaches via other avenues. There exists an extensive literature on particular values of $\varphi(q)$. See our paper \cite{BerndtRebak} for a survey of some of these values.

As to be expected, evaluations for $\varphi(e^{-\pi\sqrt{n}})$ become more elaborate with increasing $n$.  Consequently, there may be multiple ways to record these values.  In choosing a formulation, we primarily considered two options.  If our evaluation was also obtained by Ramanujan, we use his representation.  In our evaluations, we employ general theorems found by Ramanujan or by the present authors.  We therefore usually use the form of the evaluation generated by these theorems.

\section{A cubic analogue of Entry~\ref{entry:Ramanujan}}

We present a cubic analogue of Entry~\ref{entry:Ramanujan}. Observe that the definition \eqref{def:3:u} in Theorem~\ref{thm:maincubic} corresponds to \eqref{def:u,v,w}, and parts \eqref{5:1+u+v}--\eqref{5:r} are matching in both statements.

\begin{theorem}\label{thm:maincubic}
\leqnomode
For $|q|<1$, let
\begin{align}\label{def:3:u}
u := \frac{2q^{1/3}f(q, q^5)}{\varphi(q^3)}.\tag{0}
\end{align}
Then,
\begin{equation}\label{3:1+u}
\frac{\varphi(q^{1/3})}{\varphi(q^3)} = 1 + u,\tag{i}
\end{equation}
\begin{equation}\label{3:p}
p := u = \frac{2q^{1/3}(-q; q^2)_\infty}{(-q^3; q^{6})_\infty^3} = \frac{2q^{1/3}\chi(q)}{\chi^3(q^3)},\tag{ii}
\end{equation}
and
\begin{equation}\label{3:phi-to-p}
\frac{\varphi^4(q)}{\varphi^4(q^3)} = 1 + p^3.\tag{iii}
\end{equation}
Moreover,
\begin{equation}\label{3:u}
u = (\alpha p)^{1/3},\tag{iv}
\end{equation}
where $\alpha$ is a root of the equation
\begin{equation}\label{3:r}
\xi - p^2 = 0.\tag{v}
\end{equation}
\reqnomode
\end{theorem}

The essence of Theorem~\ref{thm:maincubic} can be captured in the formulation
\begin{equation*}
\frac{\varphi(q^{1/3})}{\varphi(q^3)} = 1 + \bigg(\frac{\varphi^4(q)}{\varphi^4(q^3)} - 1\bigg)^{1/3},
\end{equation*}
which is stated in Entry~1(iii) of Chapter~20 of Ramanujan's second notebook~\cite[p.~241]{RamanujanEarlierII}, \cite[pp.~345--349]{BerndtIII}.

\begin{proof}
To prove \eqref{3:1+u}, we use the identity \cite[p.~200]{RamanujanEarlierII}, \cite[p.~49, Corollary~(i)]{BerndtIII}
\begin{equation*}
\varphi(q) = \varphi(q^9) + 2qf(q^3, q^{15}).
\end{equation*}
Replace $q$ by $q^{1/3}$ and rearrange to conclude \eqref{3:1+u}.

To prove \eqref{3:p}, first apply the Jacobi triple product identity \eqref{eq:Jacobi-triple-product} to obtain the representations
\begin{align*}
f(q, q^5) &= (-q; q^6)_\infty (-q^5; q^6)_\infty (q^6; q^6)_\infty,\\
\varphi(q^3) &= f(q^3, q^3) = (-q^3; q^6)^2_\infty (q^6; q^6)_\infty.
\end{align*}
Thus,
\begin{equation}\label{eq:3:p-to-chi}
p = u = \frac{2q^{1/3}f(q, q^5)}{\varphi(q^3)} = \frac{2q^{1/3}(-q; q^6)_\infty (-q^5; q^6)_\infty}{(-q^3; q^6)^2_\infty}.
\end{equation}
Since
\begin{equation*}
(-q; q^2)_\infty = (-q; q^6)_\infty (-q^3; q^6)_\infty (-q^5; q^6)_\infty,
\end{equation*}
using \eqref{eq:3:p-to-chi}, we arrive at
\begin{align*}
p = \frac{2q^{1/3}(-q; q^2)_\infty}{(-q^3; q^{6})_\infty^3} = \frac{2q^{1/3}\chi(q)}{\chi^3(q^3)},
\end{align*}
by using the definition \eqref{def:chi} of $\chi(q)$. Thus we have proved \eqref{3:p}.

Part \eqref{3:phi-to-p} is a direct consequence of the identity \cite[p.~330, (4.6)]{BerndtV}
\begin{equation*}
\frac{\varphi^4(q)}{\varphi^4(q^3)} = 1 + 8q\frac{\chi^3(q)}{\chi^9(q^3)}
\end{equation*}
and \eqref{3:p}.

Parts \eqref{3:u} and \eqref{3:r} follow from the identity $u = p$, i.e., from \eqref{3:p}.
\end{proof}

\section{Examples for cubic identities}\label{section:cubic-examples}

To establish cubic examples, we need the values of pairs of class invariants $G_n$ and $G_{9n}$, for certain positive rational numbers $n$. Ramanujan \cite[pp.~189--199]{BerndtV} calculated the class invariant $G_{n}$ for a total of $78$ values of $n$.  Among these, there are $11$ values of $G_n$ for which $G_{9n}$ is also given. These are for $n = 1, 3, 5, 7, 9, 13, 17, 25, 37, 49,$ and $85$. In view of \eqref{eq:G}, we added the values when $n = 1/3$ to this list. For $n = 11$ and $81$, we determined the values of $G_{99}$ and $G_{729}$, which were not given by Ramanujan, and for $n = 27$, the value of $G_{243}$ is given by Watson~\cite{Watson3}. In summary, when $n$ is a positive integer or its reciprocal, we calculated the values of a total of $25$ quotients of theta functions. With the help of Theorems~\ref{thm:thetagamma} and~\ref{thm:e37}, most of them can be expressed in terms of gamma functions. The mentioned examples are summarized in Table~\ref{table:cubic}. Some of the theta functions for which we have determined values can be evaluated via other results in our paper.  These instances are marked by an asterisk in Table \ref{table:cubic}. The forms of the alternative evaluations may be different.

\begin{table}[ht]
\caption{Overview of cubic examples}\label{table:cubic}
\centering
\begin{tabular}{| r | r | r | l | r | l |}
	\hline &&&&&\\[-1em]
	$n$ & $9n$ & $3\sqrt{n}$ & Ex. for Thm.~\ref{thm:3n} & $9\sqrt{n}$ & Ex. for Thm.~\ref{thm:9n}\\
	\hline &&&&&\\[-1em]
	$1/3$ & $3$ & $\sqrt{3}$ & \eqref{eq:transform}\textsuperscript{$\ast$} & $3\sqrt{3}$ & Corollary~\ref{cor:3sqrt3} \\
	$1$ & $9$ & $3$ & Corollary~\ref{cor:3} & $9$ & Corollary~\ref{cor:9} \\
	$3$ & $27$ & $3\sqrt{3}$ & Corollary~\ref{cor:3sqrt3}\textsuperscript{$\ast$} & $9\sqrt{3}$ & Corollary~\ref{cor:9sqrt3} \\
	$5$ & $45$ & $3\sqrt{5}$ & Corollary~\ref{cor:3sqrt5} & $9\sqrt{5}$ & Corollary~\ref{cor:9sqrt5} \\
	$7$ & $63$ & $3\sqrt{7}$ & Corollary~\ref{cor:3sqrt7} & $9\sqrt{7}$ & Corollary~\ref{cor:9sqrt7} \\
	$9$ & $81$ & $9$ & Corollary~\ref{cor:9}\textsuperscript{$\ast$} & $27$ & Corollary~\ref{cor:27} \\
	$13$ & $117$ & $3\sqrt{13}$ & Corollary~\ref{cor:3sqrt13} & $9\sqrt{13}$ & Corollary~\ref{cor:9sqrt13} \\
	$17$ & $153$ & $3\sqrt{17}$ & Corollary~\ref{cor:3sqrt17} & $9\sqrt{17}$ & Corollary~\ref{cor:9sqrt17} \\
	$25$ & $225$ & $15$ & Corollary~\ref{cor:3:15} & $45$ & Corollary~\ref{cor:45} \\
	$37$ & $333$ & $3\sqrt{37}$ & Corollary~\ref{cor:3sqrt37} & $9\sqrt{37}$ & Corollary~\ref{cor:9sqrt37} \\
	$49$ & $441$ & $21$ & Corollary~\ref{cor:21} & $63$ & Corollary~\ref{cor:63} \\
	$85$ & $765$ & $3\sqrt{85}$ & Corollary~\ref{cor:3sqrt85} & $9\sqrt{85}$ & Corollary~\ref{cor:9sqrt85} \\
	\hline &&&&&\\[-1em]
	$11$ & $99$ & $3\sqrt{11}$ & Corollary~\ref{cor:3sqrt11} & $9\sqrt{11}$ & Corollary~\ref{cor:9sqrt11} \\
	$27$ & $243$ & $9\sqrt{3}$ & Corollary~\ref{cor:9sqrt3}\textsuperscript{$\ast$} & $27\sqrt{3}$ & Corollary~\ref{cor:27sqrt3} \\
	$81$ & $729$ & $27$ & Corollary~\ref{cor:27}\textsuperscript{$\ast$} & $81$ & Corollary~\ref{cor:81} \\
	\hline
\end{tabular}
\end{table}

We conclude this section with two further examples, when $n$ is a positive rational number that is not an integer or its reciprocal. These are for $n = 5/9$ given in Corollary~\ref{cor:sqrt5s3}, and for $n = 7/9$ given in Corollary~\ref{cor:sqrt7s3}.

Throughout this section, we use the definitions of Theorem~\ref{thm:maincubic}.

\begin{lemma}\label{lemma:3:p} If $q = e^{-\pi\sqrt{n}}$, for each positive rational number $n$, then
\begin{equation*}
p = \frac{\sqrt{2} G_n}{G_{9n}^3}.
\end{equation*}
\end{lemma}

\begin{proof}
From Theorem~\ref{thm:maincubic}\eqref{3:p} and \eqref{def:G},
\begin{equation*}
p = \frac{2q^{1/3}\chi(q)}{\chi^3(q^3)} = \frac{\sqrt{2} G_n}{G_{9n}^3}.\tag*{\qedhere}
\end{equation*}
\end{proof}

\begin{theorem}\label{thm:3n} If $n$ is a positive rational number, then
\begin{equation*}
\frac{\varphi(e^{-3\pi\sqrt{n}})}{\varphi(e^{-\pi\sqrt{n}})} = \frac{1}{\sqrt{3}}\bigg(1 + \frac{2\sqrt{2} G_{9n}^3}{G_n^9}\bigg)^{1/4}.
\end{equation*}
\end{theorem}

Theorem~\ref{thm:3n} is stated in \cite[p.~330, (4.5)]{BerndtV}, \cite[(3.10)]{BerndtChan}.

\begin{proof}
By using Theorem~\ref{thm:maincubic}\eqref{3:phi-to-p} with $q = e^{-\pi\sqrt{n}}$, and by Lemma~\ref{lemma:3:p}, we find that
\begin{equation}\label{eq:3sqrtn}
\frac{\varphi^4(e^{-\pi\sqrt{n}})}{\varphi^4(e^{-3\pi\sqrt{n}})} = 1 + \frac{2\sqrt{2} G_n^3}{G_{9n}^9}.
\end{equation}
With the substitution $n \mapsto (9n)^{-1}$, \eqref{eq:G}, and two applications of the transformation formula \eqref{eq:transform}, after rearrangement, we complete the proof.
\end{proof}

\begin{theorem}\label{thm:9n} If $n$ is a positive rational number, then
\begin{equation*}
\frac{\varphi(e^{-9\pi\sqrt{n}})}{\varphi(e^{-\pi\sqrt{n}})} = \frac{1}{3}\bigg(1 + \frac{\sqrt{2}G_{9n}}{G_n^3}\bigg).
\end{equation*}
\end{theorem}

Theorem~\ref{thm:9n} is stated in \cite[p.~334, (5.7)]{BerndtV}, \cite[(3.30)]{BerndtChan}.

\begin{proof}
After combining Theorem~\ref{thm:maincubic}\eqref{3:1+u} and Theorem~\ref{thm:maincubic}\eqref{3:p} with $q=e^{-\pi\sqrt{n}}$, and using Lemma~\ref{lemma:3:p}, we finish the proof in the same manner as in the proof of Theorem~\ref{thm:3n}.
\end{proof}

\begin{corollary}\label{cor:3sqrt3} We have
\begin{equation*}
\varphi(e^{-3\pi\sqrt{3}}) = 3^{-3/4}\varphi(e^{-\pi\sqrt{3}})\{1 + 2^{1/3}\} = 3^{-3/4}\varphi(e^{-\pi\sqrt{3}})\Bigg\{1 + \bigg(\frac{1}{\cos\frac{\pi}{3}}\bigg)^{1/3}\Bigg\}.
\end{equation*}
\end{corollary}

Note that Corollary~\ref{cor:3sqrt3} is the cubic analogue of Theorem~\ref{thm:enigmatic}.

\begin{proof}
We apply Theorem~\ref{thm:9n} with $n = 1/3$. From \cite[p.~189]{BerndtV}, with the use of \eqref{eq:G},
\begin{equation}\label{G3}
G_{1/3} = G_3 = 2^{1/12}.
\end{equation}
Thus,
\begin{equation*}
\frac{\varphi(e^{-9\pi/\sqrt{3}})}{\varphi(e^{-\pi/\sqrt{3}})} = \frac{1}{3}\bigg(1 + \frac{\sqrt{2}\cdot 2^{1/12}}{2^{1/4}}\bigg) = \frac{1}{3}(1 + 2^{1/3}) .
\end{equation*}
Finally, using the transformation formula \eqref{eq:transform} twice, the first representation follows. The second form is obtained by $\cos(\pi/3) = 1/2$.
\end{proof}

By using Theorem~\ref{thm:3n} with $n = 3$ and with the values $G_{3}$ from \eqref{G3} and $G_{27}$ from \eqref{G27}, we obtain
\begin{equation*}
\frac{\varphi(e^{-3\pi\sqrt{3}})}{\varphi(e^{-\pi\sqrt{3}})} = \frac{1}{\sqrt{3}}\bigg(\frac{2^{1/3} + 1}{2^{1/3} - 1}\bigg)^{1/4}.
\end{equation*}
Another representation for the value $\varphi(e^{-\pi\sqrt{3}})/\varphi(e^{-3\pi\sqrt{3}})$ was established by Jinhee Yi \cite[Theorem~4.10(iii)]{Yi}.

\begin{corollary}\label{cor:3} We have
\begin{equation*}
\frac{\varphi(e^{-3\pi})}{\varphi(e^{-\pi})} = \frac{1}{(6\sqrt{3} - 9)^{1/4}}.
\end{equation*}
\end{corollary}

Corollary~\ref{cor:3} was recorded by Ramanujan in his first notebook \cite[p.~284]{RamanujanEarlierI}, \cite[pp.~327--328]{BerndtV}, and was first proved in print by Heng Huat Chan and the first author \cite{BerndtChan}.

\begin{proof}
Invoke Theorem~\ref{thm:3n} for $n = 1$. From \cite[p.~189]{BerndtV}, $G_1 = 1$ and
\begin{equation}\label{G9}
G_9 = \bigg(\frac{1 + \sqrt{3}}{\sqrt{2}}\bigg)^{1/3}.
\end{equation}
Thus,
\begin{equation}\label{e3}
\frac{\varphi(e^{-3\pi})}{\varphi(e^{-\pi})} = \frac{1}{\sqrt{3}}\bigg(1 + 2\sqrt{2}\bigg(\frac{1 + \sqrt{3}}{\sqrt{2}}\bigg)\bigg)^{1/4} = \bigg(\frac{3 + 2\sqrt{3}}{9}\bigg)^{1/4} = \frac{1}{(6\sqrt{3} - 9)^{1/4}}.%\tag*{\qedhere}
\end{equation}
\end{proof}

\begin{corollary}\label{cor:9} We have
\begin{equation*}
\frac{\varphi(e^{-9\pi})}{\varphi(e^{-\pi})} = \frac{1 + (2(\sqrt{3} + 1))^{1/3}}{3}.
\end{equation*}
\end{corollary}

\begin{proof}
We apply Theorem~\ref{thm:9n} when $n = 1$. Using the value $G_1=1$ and the value of $G_9$ from \eqref{G9}, we find that
\begin{equation*}
\frac{\varphi(e^{-9\pi})}{\varphi(e^{-\pi})} = \frac{1}{3}\bigg(1 + \sqrt{2}\bigg(\frac{1 + \sqrt{3}}{\sqrt{2}}\bigg)^{1/3}\bigg) = \frac{1 + (2(\sqrt{3} + 1))^{1/3}}{3}.\tag*{\qedhere}
\end{equation*}
\end{proof}

Corollary~\ref{cor:9} can be found in Ramanujan's first notebook \cite[p.~287]{RamanujanEarlierI}, \cite[p.~328]{BerndtV}, and was first proved in \cite{BerndtChan}. Another representation for $\varphi(e^{-9\pi})/\varphi(e^{-\pi})$ can be obtained by using Theorem~\ref{thm:3n} with $n = 9$ and with the values $G_{9}$ from \eqref{G9} and $G_{81}$ from \eqref{G81}, combined with the value in Corollary~\ref{cor:3}. This value was also established in \cite[Theorem~4.2(ii)]{Yi2}, \cite[(6.2)]{BerndtRebak}.

\begin{corollary}\label{cor:9sqrt3} We have
\begin{equation*}
\frac{\varphi(e^{-9\pi\sqrt{3}})}{\varphi(e^{-\pi\sqrt{3}})} = \frac{1}{3}\bigg(1 + \bigg(\frac{2}{2^{1/3} - 1}\bigg)^{1/3}\bigg).
\end{equation*}
\end{corollary}

\begin{proof}
We appeal to Theorem~\ref{thm:9n} with $n = 3$. Recall that $G_3$ is provided by \eqref{G3}. We also know that \cite[p.~190]{BerndtV}
\begin{equation}\label{G27}
G_{27} = 2^{1/12}(2^{1/3} - 1)^{-1/3}.
\end{equation}
Thus,
\begin{equation*}
\frac{\varphi(e^{-9\pi\sqrt{3}})}{\varphi(e^{-\pi\sqrt{3}})} = \frac{1}{3}\bigg(1 + \frac{\sqrt{2}\cdot2^{1/12}(2^{1/3} - 1)^{-1/3}}{2^{1/4}}\bigg) = \frac{1}{3}\bigg(1 + \bigg(\frac{2}{2^{1/3} - 1}\bigg)^{1/3}\bigg).\tag*{\qedhere}
\end{equation*}
\end{proof}

\begin{corollary}\label{cor:3sqrt5} We have
\begin{equation*}
\frac{\varphi(e^{-3\pi\sqrt{5}})}{\varphi(e^{-\pi\sqrt{5}})} = \frac{(1 + 2(\sqrt{3} + \sqrt{5}))^{1/4}}{\sqrt{3}}.
\end{equation*}
\end{corollary}

Corollary~\ref{cor:3sqrt5} is given in~\cite[p.~151]{BorweinBrothersPiAGM}.

\begin{proof}
We apply Theorem~\ref{thm:3n} with $n = 5$. From \cite[p.~189]{BerndtV},
\begin{equation}\label{G5}
G_5 = \bigg(\frac{1 + \sqrt{5}}{2}\bigg)^{1/4} = (2 + \sqrt{5})^{1/12}
\end{equation}
and \cite[p.~191]{BerndtV}
\begin{equation}\label{G45}
G_{45} = (2 + \sqrt{5})^{1/4}\bigg(\frac{\sqrt{3} + \sqrt{5}}{\sqrt{2}}\bigg)^{1/3}.
\end{equation}
The desired result now follows after simplification.
\end{proof}

\begin{corollary}\label{cor:9sqrt5} We have
\begin{equation*}
\frac{\varphi(e^{-9\pi\sqrt{5}})}{\varphi(e^{-\pi\sqrt{5}})} = \frac{1 + (2(\sqrt{3} + \sqrt{5}))^{1/3}}{3}.
\end{equation*}
\end{corollary}

\begin{proof}
We apply Theorem~\ref{thm:9n} for $n = 5$. Using the values of $G_5$ and $G_{45}$ in \eqref{G5} and \eqref{G45}, respectively, we complete the proof as in Corollary~\ref{cor:3sqrt5}.
\end{proof}

\begin{corollary}\label{cor:3sqrt7} We have
\begin{equation*}
\frac{\varphi(e^{-3\pi\sqrt{7}})}{\varphi(e^{-\pi\sqrt{7}})} = \frac{1}{\sqrt{3}}\left(1 + \bigg(\frac{\sqrt{7} + \sqrt{3}}{2}\bigg)\left(\sqrt{\frac{5 + \sqrt{21}}{8}} + \sqrt{\frac{\sqrt{21} - 3}{8}}\right)^3\right)^{1/4}.
\end{equation*}
\end{corollary}

\begin{proof}
Apply Theorem~\ref{thm:3n} for $n = 7$. From \cite[p.~189]{BerndtV},
\begin{equation}\label{G7}
G_7 = 2^{1/4},
\end{equation}
and from \cite[p.~192]{BerndtV},
\begin{equation*}
G_{63} = 2^{1/4}\bigg(\frac{5 + \sqrt{21}}{2}\bigg)^{1/6}\left(\sqrt{\frac{5 + \sqrt{21}}{8}} + \sqrt{\frac{\sqrt{21} - 3}{8}}\right).
\end{equation*}
Note that
\begin{equation*}
\bigg(\frac{5 + \sqrt{21}}{2}\bigg)^{1/6} = \bigg(\frac{\sqrt{7} + \sqrt{3}}{2}\bigg)^{1/3}.
\end{equation*}
The completion of the proof now follows.
\end{proof}

\begin{corollary}\label{cor:9sqrt7} We have
\begin{equation*}
\frac{\varphi(e^{-9\pi\sqrt{7}})}{\varphi(e^{-\pi\sqrt{7}})} = \frac{1}{3}\left(1 + \bigg(\frac{\sqrt{7} + \sqrt{3}}{2}\bigg)^{1/3}\left(\sqrt{\frac{5 + \sqrt{21}}{8}} + \sqrt{\frac{\sqrt{21} - 3}{8}}\right)\right).
\end{equation*}
\end{corollary}

\begin{proof}
We apply Theorem~\ref{thm:9n} for $n = 7$. The remainder of the proof is similar to the proof of Corollary~\ref{cor:3sqrt7}.
\end{proof}

\begin{corollary}\label{cor:27} We have
\begin{equation*}
\frac{\varphi(e^{-27\pi})}{\varphi(e^{-\pi})} = \frac{1}{3(6\sqrt{3} - 9)^{1/4}}\left(1 + (\sqrt{3} - 1)\left(\frac{(2(\sqrt{3} + 1))^{1/3} + 1}{(2(\sqrt{3} - 1))^{1/3} - 1}\right)^{1/3}\right).
\end{equation*}
\end{corollary}

Corollary~\ref{cor:27} can be found in \cite{BerndtChan}.

\begin{proof}
Invoke Theorem~\ref{thm:9n} with $n = 9$. We know that
 \cite[p.~193]{BerndtV}
\begin{equation}\label{G81}
G_{81} = \left(\frac{(2(\sqrt{3} + 1))^{1/3} + 1}{(2(\sqrt{3} - 1))^{1/3} - 1}\right)^{1/3}.
\end{equation} Using \eqref{G81} and \eqref{G9} with Corollary~\ref{cor:3}, we complete the proof.
\end{proof}

\begin{corollary}\label{cor:3sqrt13} We have
\begin{align*}
\frac{\varphi(e^{-3\pi\sqrt{13}})}{\varphi(e^{-\pi\sqrt{13}})} &= \frac{1}{\sqrt{3}}\Bigg(1 + 2\sqrt{2} \bigg(\frac{\sqrt{13} - 3}{2}\bigg)^{3/2}(2\sqrt{3} + \sqrt{13})^{1/2}\\
&\qquad\qquad\times\bigg(\sqrt{208 + 120\sqrt{3}} + \sqrt{207 + 120\sqrt{3}}\bigg)^{1/2}\Bigg)^{1/4}.
\end{align*}
\end{corollary}

\begin{proof}
Employ Theorem~\ref{thm:3n} when $n = 13$. From \cite[p.~190]{BerndtV},
\begin{equation}\label{G13}
G_{13} = \bigg(\frac{\sqrt{13} + 3}{2}\bigg)^{1/4} = \bigg(\frac{\sqrt{13} - 3}{2}\bigg)^{-1/4},
\end{equation}
and from \cite[p.~193]{BerndtV},
\begin{equation*}
G_{117} = \bigg(\frac{\sqrt{13} + 3}{2}\bigg)^{1/4}(2\sqrt{3} + \sqrt{13})^{1/6}\left(\frac{3^{1/4} + \sqrt{4 + \sqrt{3}}}{2}\right).
\end{equation*}
Using the identity,
\begin{equation*}
\left(\frac{3^{1/4} + \sqrt{4 + \sqrt{3}}}{2}\right)^3 = \bigg(\sqrt{208 + 120\sqrt{3}} + \sqrt{207 + 120\sqrt{3}}\bigg)^{1/2},
\end{equation*}
we can complete the proof.
\end{proof}

\begin{corollary}\label{cor:9sqrt13} We have
\begin{equation*}
\frac{\varphi(e^{-9\pi\sqrt{13}})}{\varphi(e^{-\pi\sqrt{13}})} = \frac{1}{3}\left(1 + \sqrt{2}\bigg(\frac{\sqrt{13} - 3}{2}\bigg)^{1/2}(2\sqrt{3} + \sqrt{13})^{1/6}\left(\frac{3^{1/4} + \sqrt{4 + \sqrt{3}}}{2}\right)\right).
\end{equation*}
\end{corollary}

\begin{proof}
Apply Theorem~\ref{thm:9n} for $n = 13$. The remainder of the proof is analogous to the proof of Corollary~\ref{cor:3sqrt13}.
\end{proof}

\begin{corollary}\label{cor:3sqrt17} We have
\begin{align*}
\frac{\varphi(e^{-3\pi\sqrt{17}})}{\varphi(e^{-\pi\sqrt{17}})} &= \frac{1}{\sqrt{3}}\left(1 + 2\sqrt{2}\left(\sqrt{\frac{5 + \sqrt{17}}{8}} - \sqrt{\frac{\sqrt{17} - 3}{8}}\right)^3\right.\\
&\qquad\qquad\times\left.\left(\sqrt{\frac{37 + 9\sqrt{17}}{4}} + \sqrt{\frac{33 + 9\sqrt{17}}{4}}\right)\vphantom{\left(\sqrt{\frac{5 + \sqrt{17}}{8}} - \sqrt{\frac{\sqrt{17} - 3}{8}}\right)^3}\right)^{1/4}.
\end{align*}
\end{corollary}

\begin{proof}
Invoke Theorem~\ref{thm:3n} for $n = 17$. Now, from \cite[p.~190]{BerndtV},
\begin{equation*}
G_{17} = \sqrt{\frac{5 + \sqrt{17}}{8}} + \sqrt{\frac{\sqrt{17} - 3}{8}} = \left(\sqrt{\frac{5 + \sqrt{17}}{8}} - \sqrt{\frac{\sqrt{17} - 3}{8}}\right)^{-1},
\end{equation*}
and from \cite[p.~194]{BerndtV},
\begin{equation*}
G_{153} = \left(\sqrt{\frac{5 + \sqrt{17}}{8}} + \sqrt{\frac{\sqrt{17} - 3}{8}}\right)^2\left(\sqrt{\frac{37 + 9\sqrt{17}}{4}} + \sqrt{\frac{33 + 9\sqrt{17}}{4}}\right)^{1/3}.
\end{equation*}
The proof is now straightforwardly completed.
\end{proof}

\begin{corollary}\label{cor:9sqrt17} We have
\begin{align*}
\frac{\varphi(e^{-9\pi\sqrt{17}})}{\varphi(e^{-\pi\sqrt{17}})} &= \frac{1}{3}\left(1 + \sqrt{2}\left(\sqrt{\frac{5 + \sqrt{17}}{8}} - \sqrt{\frac{\sqrt{17} - 3}{8}}\right)\right.\\
&\qquad\qquad\times\left.\left(\sqrt{\frac{37 + 9\sqrt{17}}{4}} + \sqrt{\frac{33 + 9\sqrt{17}}{4}}\right)^{1/3}\right).
\end{align*}
\end{corollary}

\begin{proof}
Apply Theorem~\ref{thm:9n} for $n = 17$. The remainder of the proof is similar to that of Corollary~\ref{cor:3sqrt17}.
\end{proof}

\begin{corollary}\label{cor:3:15} We have
\begin{align*}
\frac{\varphi(e^{-15\pi})}{\varphi(e^{-\pi})} &= \frac{1}{\sqrt{3}(5\sqrt{5} - 10)^{1/2}}\left(1 + 2\sqrt{2}(\sqrt{5} - 2)^2(2 + \sqrt{3})\vphantom{\bigg(\sqrt{4275 + 1104\sqrt{15}} + \sqrt{4276 + 1104\sqrt{15}}\bigg)^{1/2}}\right.\\
&\qquad\qquad\times\left.\bigg(\sqrt{4276 + 1104\sqrt{15}} + \sqrt{4275 + 1104\sqrt{15}}\bigg)^{1/2}\right)^{1/4}.
\end{align*}
\end{corollary}

\begin{proof}
Apply Theorem~\ref{thm:3n} with $n = 25$. Now, from \cite[p.~190]{BerndtV},
\begin{equation}\label{G25}
G_{25} = \frac{\sqrt{5} + 1}{2} = \bigg(\frac{\sqrt{5} - 1}{2}\bigg)^{-1},
\end{equation}
and from \cite[p.~195]{BerndtV},
\begin{equation*}
G_{225} = \bigg(\frac{\sqrt{5} + 1}{2}\bigg)(2 + \sqrt{3})^{1/3}\left(\frac{\sqrt{4 + \sqrt{15}} + (15)^{1/4}}{2}\right).
\end{equation*}
Note that
\begin{equation*}
\bigg(\frac{\sqrt{5} - 1}{2}\bigg)^6 = (\sqrt{5} - 2)^2
\end{equation*}
and
\begin{equation*}
\left(\frac{\sqrt{4 + \sqrt{15}} + (15)^{1/4}}{2}\right)^3 = \bigg(\sqrt{4276 + 1104\sqrt{15}} + \sqrt{4275 + 1104\sqrt{15}}\bigg)^{1/2}.
\end{equation*}
Using the evaluation \cite{BerndtChan}, \cite[p.~327]{BerndtV}
\begin{equation*}
\frac{\varphi(e^{-5\pi})}{\varphi(e^{-\pi})} = \frac{1}{(5\sqrt{5} - 10)^{1/2}},
\end{equation*}
which is also stated in Corollary~\ref{cor:5}, we complete the proof.
\end{proof}

\begin{corollary}\label{cor:45} We have
\begin{equation*}
\frac{\varphi(e^{-45\pi})}{\varphi(e^{-\pi})} = \frac{1}{3(5\sqrt{5} - 10)^{1/2}}\left(1 + \sqrt{2}\bigg(\frac{3 - \sqrt{5}}{2}\bigg)(2 + \sqrt{3})^{1/3}\left(\frac{\sqrt{4 + \sqrt{15}} + (15)^{1/4}}{2}\right)\right)\!.
\end{equation*}
\end{corollary}

\begin{proof}
We appeal to Theorem~\ref{thm:9n} for $n = 25$. The remainder of the proof is analogous to the proof of Corollary~\ref{cor:3:15}.
\end{proof}

An alternative representation for $\varphi(e^{-45\pi})/\varphi(e^{-\pi})$ was established by Ramanujan in his first notebook~\cite[p.~312]{RamanujanEarlierI}, \cite{BerndtChan}, \cite[p.~328]{BerndtV}, namely,
\begin{equation*}
\frac{\varphi(e^{-45\pi})}{\varphi(e^{-\pi})} = \frac{3 + \sqrt{5} + (\sqrt{3} + \sqrt{5} + (60)^{1/4})(2 + \sqrt{3})^{1/3}}{3(10 + 10\sqrt{5})^{1/2}}.
\end{equation*}

\begin{corollary}\label{cor:3sqrt37} We have
\begin{align*}
\frac{\varphi(e^{-3\pi\sqrt{37}})}{\varphi(e^{-\pi\sqrt{37}})} &= \frac{1}{\sqrt{3}}\Bigg(1 + 2\sqrt{2}(\sqrt{37} - 6)^{3/2}(7\sqrt{3} + 2\sqrt{37})^{1/2}\\
&\qquad\qquad\times\bigg(\sqrt{42193 + 24360\sqrt{3}} + \sqrt{42192 + 24360\sqrt{3}}\bigg)^{1/2}\Bigg)^{1/4}.
\end{align*}
\end{corollary}

\begin{proof}
We utilize Theorem~\ref{thm:3n} with $n = 37$. From \cite[p.~191]{BerndtV},
\begin{equation*}
G_{37} = (\sqrt{37} + 6)^{1/4} = (\sqrt{37} - 6)^{-1/4},
\end{equation*}
and from \cite[p.~196]{BerndtV},
\begin{equation*}
G_{333} = (\sqrt{37} + 6)^{1/4}(7\sqrt{3} + 2\sqrt{37})^{1/6}\left(\frac{\sqrt{7 + 2\sqrt{3}} + \sqrt{3 + 2\sqrt{3}}}{2}\right).
\end{equation*}
With the use of the identity
\begin{equation*}
\left(\frac{\sqrt{7 + 2\sqrt{3}} + \sqrt{3 + 2\sqrt{3}}}{2}\right)^3 = \bigg(\sqrt{42193 + 24360\sqrt{3}} + \sqrt{42192 + 24360\sqrt{3}}\bigg)^{1/2},
\end{equation*}
we can complete the proof.
\end{proof}

\begin{corollary}\label{cor:9sqrt37} We have
\begin{equation*}
\frac{\varphi(e^{-9\pi\sqrt{37}})}{\varphi(e^{-\pi\sqrt{37}})} = \frac{1}{3}\left(1 + \sqrt{2}(\sqrt{37} - 6)^{1/2}(7\sqrt{3} + 2\sqrt{37})^{1/6}\left(\frac{\sqrt{7 + 2\sqrt{3}} + \sqrt{3 + 2\sqrt{3}}}{2}\right)\right).
\end{equation*}
\end{corollary}

\begin{proof}
Apply Theorem~\ref{thm:9n} for $n = 37$. The remainder of the proof is similar to the proof of Corollary~\ref{cor:3sqrt37}.
\end{proof}

\begin{corollary}\label{cor:21} We have
\begin{align*}
\frac{\varphi(e^{-21\pi})}{\varphi(e^{-\pi})} &= \frac{1}{\sqrt{3}}\left(\frac{\sqrt{13 + \sqrt{7}} + \sqrt{7 + 3\sqrt{7}}}{14}(28)^{1/8}\right)^{1/2}\\
&\qquad\times\left\{1 + 2\sqrt{2}\bigg(\sqrt{932 + 352\sqrt{7}} - \sqrt{931 + 352\sqrt{7}}\bigg)\vphantom{\left(\frac{\sqrt{3 + \sqrt{7}} + (6\sqrt{7})^{1/4}}{\sqrt{3 + \sqrt{7}} - (6\sqrt{7})^{1/4}}\right)^{3/2}}\right.\\
&\qquad\qquad\times\left.\bigg(\frac{\sqrt{3} + \sqrt{7}}{2}\bigg)^{3/2}(2 + \sqrt{3})^{1/2}\left(\frac{\sqrt{3 + \sqrt{7}} + (6\sqrt{7})^{1/4}}{\sqrt{3 + \sqrt{7}} - (6\sqrt{7})^{1/4}}\right)^{3/2}\right\}^{1/4}.
\end{align*}
\end{corollary}

Another representation for $\varphi(e^{-21\pi})/\varphi(e^{-\pi})$ is given in \cite[Theorem~6.4]{Rebak}.

\begin{proof}
Invoke Theorem~\ref{thm:3n} for $n = 49$. We know from \cite[p.~191]{BerndtV} that
\begin{equation}\label{G49}
G_{49} = \frac{\sqrt{4 + \sqrt{7}} + 7^{1/4}}{2} = \left(\frac{\sqrt{4 + \sqrt{7}} - 7^{1/4}}{2}\right)^{-1},
\end{equation}
and from \cite[p.~197]{BerndtV} that
\begin{equation*}
G_{441} = \sqrt{\frac{\sqrt{3} + \sqrt{7}}{2}}(2 + \sqrt{3})^{1/6}\sqrt{\frac{2 + \sqrt{7} + \sqrt{7 + 4\sqrt{7}}}{2}}\sqrt{\frac{\sqrt{3 + \sqrt{7}} + (6\sqrt{7})^{1/4}}{\sqrt{3 + \sqrt{7}} - (6\sqrt{7})^{1/4}}}.
\end{equation*}
Note that
\begin{equation*}
\frac{\sqrt{4 + \sqrt{7}} + 7^{1/4}}{2} = \sqrt{\frac{2 + \sqrt{7} + \sqrt{7 + 4\sqrt{7}}}{2}}
\end{equation*}
and
\begin{equation*}
\left(\frac{\sqrt{4 + \sqrt{7}} - 7^{1/4}}{2}\right)^6 = \sqrt{932 + 352\sqrt{7}} - \sqrt{931 + 352\sqrt{7}}.
\end{equation*}
Using the evaluation \cite[p.~297]{RamanujanEarlierI}, \cite{BerndtChan}, \cite[p.~328]{BerndtV}, \cite{Rebak}
\begin{equation}\label{e7}
\frac{\varphi^2(e^{-7\pi})}{\varphi^2(e^{-\pi})} = \frac{\sqrt{13 + \sqrt{7}} + \sqrt{7 + 3\sqrt{7}}}{14}(28)^{1/8},
\end{equation}
we have all the ingredients to complete the proof.
\end{proof}

\begin{corollary}\label{cor:63} We have
\begin{align*}
\frac{\varphi(e^{-63\pi})}{\varphi(e^{-\pi})} &= \frac{1}{3}\left(\frac{\sqrt{13 + \sqrt{7}} + \sqrt{7 + 3\sqrt{7}}}{14}(28)^{1/8}\right)^{1/2}\\
&\qquad\times\left\{1 + \sqrt{2}\left(\frac{2 + \sqrt{7} - \sqrt{7 + 4\sqrt{7}}}{2}\right)\vphantom{\left(\frac{\sqrt{3 + \sqrt{7}} + (6\sqrt{7})^{1/4}}{\sqrt{3 + \sqrt{7}} - (6\sqrt{7})^{1/4}}\right)^{1/2}}\right.\\
&\qquad\qquad\times\left.\bigg(\frac{\sqrt{3} + \sqrt{7}}{2}\bigg)^{1/2}(2 + \sqrt{3})^{1/6}\left(\frac{\sqrt{3 + \sqrt{7}} + (6\sqrt{7})^{1/4}}{\sqrt{3 + \sqrt{7}} - (6\sqrt{7})^{1/4}}\right)^{1/2}\right\}.
\end{align*}
\end{corollary}

Corollary~\ref{cor:63} can be found in \cite{BerndtChan}, in a slightly different form.

\begin{proof}
We apply Theorem~\ref{thm:9n} for $n = 49$. Note that
\begin{equation*}
\sqrt{\frac{2 + \sqrt{7} + \sqrt{7 + 4\sqrt{7}}}{2}} = \left(\frac{2 + \sqrt{7} - \sqrt{7 + 4\sqrt{7}}}{2}\right)^{-1/2}.
\end{equation*}
The remainder of the proof is similar to the proof of Corollary~\ref{cor:21}.
\end{proof}

\begin{corollary}\label{cor:3sqrt85} We have
\begin{multline*}
\frac{\varphi(e^{-3\pi\sqrt{85}})}{\varphi(e^{-\pi\sqrt{85}})} = \frac{1}{\sqrt{3}}\left(1 + 2\sqrt{2}(\sqrt{5} - 2)^2\bigg(\frac{\sqrt{85} - 9}{2}\bigg)^{3/2}(16 + \sqrt{255})^{1/4}(4 + \sqrt{15})^{3/4}\vphantom{\left(\sqrt{\frac{6 + \sqrt{51}}{4}} + \sqrt{\frac{10 + \sqrt{51}}{4}}\right)^{3/2}}\right.\\
\times\left.\left(\sqrt{\frac{6 + \sqrt{51}}{4}} + \sqrt{\frac{10 + \sqrt{51}}{4}}\right)^{3/2}\left(\sqrt{\frac{18 + 3\sqrt{51}}{4}} + \sqrt{\frac{22 + 3\sqrt{51}}{4}}\right)^{3/2}\right)^{1/4}.
\end{multline*}
\end{corollary}

\begin{proof}
We apply Theorem~\ref{thm:3n} for $n = 85$. We know that \cite[p.~193]{BerndtV}
\begin{equation*}
G_{85} = \bigg(\frac{\sqrt{5} + 1}{2}\bigg)\bigg(\frac{\sqrt{85} + 9}{2}\bigg)^{1/4} = \bigg(\frac{\sqrt{5} - 1}{2}\bigg)^{-1}\bigg(\frac{\sqrt{85} - 9}{2}\bigg)^{-1/4}
\end{equation*}
and \cite[p.~198]{BerndtV}
\begin{align*}
G_{765} &= \sqrt{\frac{3 + \sqrt{5}}{2}}(16 + \sqrt{255})^{1/12}(4 + \sqrt{15})^{1/4}\bigg(\frac{\sqrt{85} + 9}{2}\bigg)^{1/4}\\
&\qquad\times\left(\sqrt{\frac{6 + \sqrt{51}}{4}} + \sqrt{\frac{10 + \sqrt{51}}{4}}\right)^{1/2}\left(\sqrt{\frac{18 + 3\sqrt{51}}{4}} + \sqrt{\frac{22 + 3\sqrt{51}}{4}}\right)^{1/2}.
\end{align*}
By noting that
\begin{equation*}
\sqrt{\frac{3 + \sqrt{5}}{2}} = \frac{\sqrt{5} + 1}{2} \quad \text{and} \quad \bigg(\frac{\sqrt{5} - 1}{2}\bigg)^6 = (\sqrt{5} - 2)^2,
\end{equation*}
we complete the proof.
\end{proof}

\begin{corollary}\label{cor:9sqrt85} We have
\begin{multline*}
\frac{\varphi(e^{-9\pi\sqrt{85}})}{\varphi(e^{-\pi\sqrt{85}})} = \frac{1}{3}\left(1 + \sqrt{2}(\sqrt{5} - 2)^{2/3}\bigg(\frac{\sqrt{85} - 9}{2}\bigg)^{1/2}(16 + \sqrt{255})^{1/12}(4 + \sqrt{15})^{1/4}\vphantom{\left(\sqrt{\frac{6 + \sqrt{51}}{4}} + \sqrt{\frac{10 + \sqrt{51}}{4}}\right)^{1/2}}\right.\\
\times\left.\left(\sqrt{\frac{6 + \sqrt{51}}{4}} + \sqrt{\frac{10 + \sqrt{51}}{4}}\right)^{1/2}\left(\sqrt{\frac{18 + 3\sqrt{51}}{4}} + \sqrt{\frac{22 + 3\sqrt{51}}{4}}\right)^{1/2}\right).
\end{multline*}
\end{corollary}

\begin{proof}
We utilize Theorem~\ref{thm:9n} when $n = 85$. The rest of the proof is similar to the proof of Corollary~\ref{cor:3sqrt85}.
\end{proof}

\begin{corollary}\label{cor:3sqrt11} We have
\begin{equation*}
\frac{\varphi(e^{-3\pi\sqrt{11}})}{\varphi(e^{-\pi\sqrt{11}})} = \frac{1}{\sqrt{3}}  \bigg(\frac{1}{6}(1 + \sqrt{33})(6(9 - \sqrt{33}))^{1/3} + \frac{2}{3}(6(9 + \sqrt{33}))^{1/3} + 3 \bigg)^{1/4}.
\end{equation*}
\end{corollary}

\begin{proof}
We apply Theorem~\ref{thm:3n} for $n = 11$. We know that~\cite[p.~189]{BerndtV}
\begin{equation*}
G_{11} = 2^{-1/4}x,
\end{equation*}
where
\begin{equation*}
x^3 - 2x^2 + 2x - 2 = 0.
\end{equation*}
Thus, by solving this cubic equation, we find that
\begin{equation*}
G_{11} = \frac{1}{3 \cdot 2^{1/4}} \bigg((3\sqrt{33} + 17)^{1/3} - (3\sqrt{33} - 17)^{1/3} + 2\bigg).
\end{equation*}

Furthermore, for each positive rational number $n$, we have~\cite[p.~145, (4.7.9)]{BorweinBrothersPiAGM}
\begin{equation}\label{eq:modular9}
\bigg(1 + \frac{2\sqrt{2} G_{9n}^3}{G_{n}^9}\bigg)\bigg(1 + \frac{2\sqrt{2} G_{n}^3}{G_{9n}^9}\bigg)  = 9.
\end{equation}
By using the equation~\eqref{eq:modular9}, or by using the formula for $G_{9n}$ given in~\cite[p.~205, Theorem~3.5]{BerndtV}, we can verify with \emph{Mathematica} that for $n = 11$,
\begin{equation*}
G_{99}^3 = \frac{2^{1/4}}{3} \bigg((7822 + 1362\sqrt{33})^{1/3} + (7957 + 1383\sqrt{33})^{1/3} + 2\sqrt{33} + 13 \bigg).
\end{equation*}

Combining these, we can check with \emph{Mathematica} that
\begin{equation*}
\frac{2\sqrt{2}G_{99}^3}{G_{11}^9} = \frac{1}{6}(1 + \sqrt{33})(6(9 - \sqrt{33}))^{1/3} + \frac{2}{3}(6(9 + \sqrt{33}))^{1/3} + 2,
\end{equation*}
and we obtain the desired result.
\end{proof}

\begin{corollary}\label{cor:9sqrt11} We have
\begin{equation*}
\frac{\varphi(e^{-9\pi\sqrt{11}})}{\varphi(e^{-\pi\sqrt{11}})} = \frac{1}{3}  \Bigg(1 + \bigg( \frac{1}{6}(1 + \sqrt{33})(6(9 - \sqrt{33}))^{1/3} + \frac{2}{3}(6(9 + \sqrt{33}))^{1/3} + 2 \bigg)^{1/3}\Bigg).
\end{equation*}
\end{corollary}

\begin{proof}
We utilize Theorem~\ref{thm:9n} when $n = 11$. The remainder of the proof is similar to the proof of Corollary~\ref{cor:3sqrt11}.
\end{proof}

\begin{corollary}\label{cor:27sqrt3} We have
\begin{equation*}
\frac{\varphi(e^{-27\pi\sqrt{3}})}{\varphi(e^{-\pi\sqrt{3}})}  = \frac{1 + 2^{1/3}}{3^{7/4}}\Bigg(1 + 2^{1/3}(2^{1/3} - 1)\frac{2^{1/3} + 2^{2/3} + 3^{1/3}}{(9 - 2 \cdot 3^{4/3})^{1/3}}\Bigg).
\end{equation*}
\end{corollary}

\begin{proof}
We apply Theorem~\ref{thm:9n} for $n = 27$. Recall that $G_{27}$ is provided by \eqref{G27}. From Watson's paper \cite[pp.~100--105]{Watson3}, we also know that
\begin{equation*}
G_{243} = \frac{2^{1/12}(2^{1/3} + 2^{2/3} + 3^{1/3})}{(9 - 2 \cdot 3^{4/3})^{1/3}}.
\end{equation*}
Using the value for $\varphi(e^{-3\pi\sqrt{3}})/\varphi(e^{-\pi\sqrt{3}})$ given in Corollary~\ref{cor:3sqrt3}, we finish the proof.
\end{proof}

An alternative form of $G_{243}$ can be obtained by using Theorem~\ref{thm:3n} with $n = 27$, together with \eqref{G27} and the values given in Corollaries~\ref{cor:3sqrt3} and~\ref{cor:9sqrt3}.

\begin{corollary}\label{cor:81} We have
\begin{align*}
&\frac{\varphi(e^{-81\pi})}{\varphi(e^{-\pi})}  = \frac{1 + (2(\sqrt{3} + 1))^{1/3}}{9}\\
&\,\times\left(1 + 2^{1/3}\left(\frac{(2(\sqrt{3} - 1))^{1/3} - 1}{(2(\sqrt{3} + 1))^{1/3} + 1}\right)^{8/9}\left(\frac{3(\sqrt{3} + 1)}{\sqrt{3} + 1 - \left(\frac{(2(\sqrt{3} + 1))^{1/3} + 1}{(2(\sqrt{3} - 1))^{1/3} - 1}\right)^{1/3}} - 2\right)^{1/3}\right).
\end{align*}
\end{corollary}

\begin{proof} We apply Theorem~\ref{thm:9n} for $n = 81$. Recall the value $G_{81}$ from \eqref{G81}. Furthermore, by using \eqref{eq:modular9}, or by using the formula for $G_{9n}$ from~\cite[p.~205, Theorem~3.5]{BerndtV}, we can verify with \emph{Mathematica} that for $n = 81$,
\begin{equation*}
G_{729}^3 = \frac{\sqrt{2}}{2}\left(\frac{(2(\sqrt{3} + 1))^{1/3} + 1}{(2(\sqrt{3} - 1))^{1/3} - 1}\right)^{1/3}\left(\frac{3(\sqrt{3} + 1)}{\sqrt{3} + 1 - \left(\frac{(2(\sqrt{3} + 1))^{1/3} + 1}{(2(\sqrt{3} - 1))^{1/3} - 1}\right)^{1/3}} - 2\right).
\end{equation*}
Using the value $\varphi(e^{-9\pi})/\varphi(e^{-\pi})$ given in Corollary~\ref{cor:9}, we complete the proof.
\end{proof}

Another expression for $G_{729}$ can be obtained by using Theorem~\ref{thm:3n} with $n = 81$, together with \eqref{G81} and the values given in Corollaries~\ref{cor:9} and~\ref{cor:27}.

In this paper we focus on theta function values $\varphi(e^{-\pi\sqrt{n}})$, when $n$ is a positive integer or its reciprocal. By using modular equations and known class invariants, we can derive further examples. In Ramanujan's second notebook \cite[p.~227]{RamanujanEarlierII}, \cite[p.~210]{BerndtIII}, we find the value given in Corollary~\ref{cor:sqrt5s3}. This value is also given with a misprint corrected in~\cite[p.~150]{BorweinBrothersPiAGM}. We finish this section with one further example of this kind.

\begin{corollary}\label{cor:sqrt5s3} We have
\begin{equation*}
(\sqrt{5} + \sqrt{3})\varphi(e^{-\pi\sqrt{5}/3}) = (3 + \sqrt{3})\varphi(e^{-3\pi\sqrt{5}}).
\end{equation*}
\end{corollary}

\begin{proof}
Apply Theorem~\ref{thm:9n} with $n = 5/9$. Recall the value \eqref{G5} of $G_5$. From \cite[p.~345]{BerndtV}, with \eqref{eq:G},
\begin{equation*}
G_{5/9} = (2 + \sqrt{5})^{1/4}\bigg(\frac{\sqrt{5} - \sqrt{3}}{\sqrt{2}}\bigg)^{1/3} = \bigg(\frac{\sqrt{5} + 1}{2}\bigg)^{3/4}\bigg(\frac{\sqrt{5} - \sqrt{3}}{\sqrt{2}}\bigg)^{1/3}.
\end{equation*}
Note that
\begin{equation*}
\bigg(\frac{1 + \sqrt{5}}{2}\bigg)^{-2} = \frac{3 - \sqrt{5}}{2} \quad \text{and} \quad \frac{\sqrt{2}G_5}{G_{5/9}^3} = \frac{3 - \sqrt{5}}{\sqrt{5} - \sqrt{3}}.
\end{equation*}
Thus, we find that
\begin{equation*}
\frac{\varphi(e^{-3\pi\sqrt{5}})}{\varphi(e^{-\pi\sqrt{5}/3})} = \frac{1}{3}\bigg(1 + \frac{3 - \sqrt{5}}{\sqrt{5} - \sqrt{3}}\bigg) = \frac{1}{3}\bigg(\frac{3 - \sqrt{3}}{\sqrt{5} - \sqrt{3}}\bigg) = \frac{\sqrt{5} + \sqrt{3}}{3 + \sqrt{3}}.\tag*{\qedhere}
\end{equation*}
\end{proof}

\begin{corollary}\label{cor:sqrt7s3} We have
\begin{equation*}
\frac{\varphi(e^{-3\pi\sqrt{7}})}{\varphi(e^{-\pi\sqrt{7}/3})} = \frac{1}{12}(3 + \sqrt{21})(1 + (2\sqrt{21} - 9)^{1/2}).
\end{equation*}
\end{corollary}

\begin{proof}
We apply Theorem~\ref{thm:9n} for $n = 7/9$. Recall the value \eqref{G7} of $G_7$. From \cite[p.~349]{BerndtV}, with \eqref{eq:G},
\begin{equation*}
G_{7/9} = 2^{1/4}\bigg(\frac{5 + \sqrt{21}}{2}\bigg)^{1/6}\left(\sqrt{\frac{5 + \sqrt{21}}{8}} - \sqrt{\frac{\sqrt{21} - 3}{8}}\right).
\end{equation*}
Observe that
\begin{equation*}
\bigg(\frac{5 + \sqrt{21}}{2}\bigg)^{1/2}\left(\sqrt{\frac{5 + \sqrt{21}}{8}} - \sqrt{\frac{\sqrt{21} - 3}{8}}\right)^3 = \frac{1}{4}\bigg(\sqrt{21} - 1 + \sqrt{6(\sqrt{21} - 3)}\bigg).
\end{equation*}
Further elementary calculations and simplifications give the result.
\end{proof}

\section{A quintic analogue of Entry~\ref{entry:Ramanujan}}

In Theorem~\ref{thm:mainquintic} we present a quintic analogue of Entry~\ref{entry:Ramanujan}. Observe that the definition \eqref{def:5:u,v} in the theorem corresponds to \eqref{def:u,v,w}, and parts \eqref{5:1+u+v}--\eqref{5:r} are matching in both statements.

\begin{lemma}\label{lemma:quintic} For $|q| < 1$,
\begin{align}
\label{quintic-lemma-first} \varphi(q^{1/5}) - \varphi(q^5) &= 2q^{1/5}f(q^3, q^7) + 2q^{4/5}f(q, q^9),\\
\label{quintic-lemma-prod} \varphi^2(q) - \varphi^2(q^5) &= 4qf(q,q^9)f(q^3,q^7),\\
\label{quintic-lemma-sum}
32qf^5(q^3,q^7) + 32q^4f^5(q,q^9) &= \bigg(\frac{\varphi^2(q)}{\varphi(q^5)} - \varphi(q^5)\bigg)\{\varphi^4(q) - 4\varphi^2(q)\varphi^2(q^5) + 11\varphi^4(q^5)\}.
\end{align}
\end{lemma}

\begin{proof}
Proofs of \eqref{quintic-lemma-first}--\eqref{quintic-lemma-sum} can be found in \cite[pp.~262--265, Entry~10 (ii), (iv), (vii)]{BerndtIII}.
\end{proof}

\begin{theorem}\label{thm:mainquintic} For $|q|<1$, let
\leqnomode
\begin{align}\label{def:5:u,v}
u := \df{2q^{1/5}f(q^3, q^7)}{\varphi(q^5)}\qquad \text{and} \qquad
v := \df{2q^{4/5}f(q, q^9)}{\varphi(q^5)}.\tag{0}
\end{align}
Then,
\begin{equation}\label{5:1+u+v}
\frac{\varphi(q^{1/5})}{\varphi(q^5)} = 1 + u + v,\tag{i}
\end{equation}
\begin{equation}\label{5:p}
p := uv = \frac{4q(-q; q^2)_\infty}{(-q^5; q^{10})_\infty^5} = \frac{4q\chi(q)}{\chi^5(q^5)},\tag{ii}
\end{equation}
and
\begin{equation}\label{5:phi-to-p}
\frac{\varphi^2(q)}{\varphi^2(q^5)} = 1 + p.\tag{iii}
\end{equation}
Furthermore,
\begin{equation}\label{5:u,v}
u = (\alpha p)^{1/5} \qquad\text{and}\qquad v = (\beta p)^{1/5},\tag{iv}
\end{equation}
where $\alpha$ and $\beta$ are the roots of the quadratic equation
\begin{equation}\label{5:r}
\xi^2 - ((p - 1)^2 + 7)\xi + p^3 = 0.\tag{v}
\end{equation}
\reqnomode
\end{theorem}

Theorem~\ref{thm:mainquintic} is a slightly rewritten version of Entry~11(i) in Chapter~19 of Ramanujan's second notebook~\cite[p.~234]{RamanujanEarlierII}, proved in \cite[pp.~265--268]{BerndtIII}. J.~M. and P.~B.~Borwein \cite[Theorem~1]{BorweinBrothersApproxPi} established a formula for $\varphi(q^{25})/\varphi(q)$ similar to~\eqref{5:1+u+v}.

\begin{proof}
Part \eqref{5:1+u+v} follows from \eqref{quintic-lemma-first} of Lemma~\ref{lemma:quintic} after rearrangement.

To prove \eqref{5:p}, we first employ the Jacobi triple product identity \eqref{eq:Jacobi-triple-product} to establish the representations,
\begin{align*}
f(q^3, q^7) &= (-q^3; q^{10})_\infty (-q^7; q^{10})_\infty (q^{10}; q^{10})_\infty,\\
f(q, q^9) &= (-q; q^{10})_\infty (-q^9; q^{10})_\infty (q^{10}; q^{10})_\infty,\\
\varphi(q^5) &= f(q^5, q^5) = (-q^5; q^{10})_\infty^2 (q^{10}; q^{10})_\infty.
\end{align*}
Thus,
\begin{equation}\label{eq:5:p-to-chi}
p = uv = \frac{4qf(q^3, q^7)f(q, q^9)}{\varphi^2(q^5)} = \frac{4q(-q^3; q^{10})_\infty (-q^7; q^{10})_\infty (-q; q^{10})_\infty (-q^9; q^{10})_\infty}{(-q^5; q^{10})_\infty^4}.
\end{equation}
Since
\begin{equation*}
(-q; q^2)_\infty = \prod_{k=1}^5 (-q^{2k - 1}; q^{10})_\infty,
\end{equation*}
using \eqref{eq:5:p-to-chi}, we arrive at
\begin{equation*}
p = \frac{4q(-q; q^2)_\infty}{(-q^5; q^{10})_\infty^5} = \frac{4q\chi(q)}{\chi^5(q^5)},
\end{equation*}
upon using the definition \eqref{def:chi} of $\chi(q)$. We have therefore proved \eqref{5:p}.

Using the definitions of $u$ and $v$ from \eqref{def:5:u,v}, and also \eqref{quintic-lemma-prod} of Lemma~\ref{lemma:quintic}, we deduce that
\begin{equation*}
p = uv = \frac{4qf(q^3, q^7)f(q, q^9)}{\varphi^2(q^5)} = \frac{\varphi^2(q) - \varphi^2(q^5)}{\varphi^2(q^5)} = \frac{\varphi^2(q)}{\varphi^2(q^5)} - 1,
\end{equation*}
which establishes \eqref{5:phi-to-p}.

To prove that \eqref{5:u,v} holds, we define $\alpha$ and $\beta$ by \eqref{5:u,v}, and then derive the coefficients in \eqref{5:r}. If $p = 0$, then $u = v = 0$. Suppose that $p \neq 0$. For the first degree term of \eqref{5:r}, using \eqref{5:u,v}, \eqref{5:phi-to-p}, and \eqref{quintic-lemma-sum} of Lemma~\ref{lemma:quintic}, we find that
\begin{align*}
\alpha + \beta &= \frac{u^5}{p} + \frac{v^5}{p} = \frac{32qf^5(q^3,q^7) + 32q^4f^5(q,q^9)}{\varphi^5(q^5) p}\\
&= \frac{\frac{\varphi^2(q)}{\varphi(q^5)} - \varphi(q^5)}{\varphi^5(q^5)\Big(\frac{\varphi^2(q)}{\varphi^2(q^5)} - 1\Big)}\big\{\varphi^4(q) - 4\varphi^2(q)\varphi^2(q^5) + 11\varphi^4(q^5)\big\}\\
&= \frac{1}{\varphi^4(q^5)}\big\{\varphi^4(q) - 4\varphi^2(q)\varphi^2(q^5) + 11\varphi^4(q^5)\big\}\\
&= \frac{\varphi^4(q)}{\varphi^4(q^5)} - 4\frac{\varphi^2(q)}{\varphi^2(q^5)} + 11\\
&= \bigg(\frac{\varphi^2(q)}{\varphi^2(q^5)} - 1\bigg)^2 - 2\bigg(\frac{\varphi^2(q)}{\varphi^2(q^5)} - 1\bigg) + 8\\
&= p^2 - 2p + 8\\
&= (p - 1)^2 + 7.
\end{align*}
For the constant term, using \eqref{5:p} and \eqref{5:u,v}, we immediately deduce that $\alpha\beta = p^3$. Alternatively, using \eqref{def:5:u,v}, \eqref{5:u,v}, \eqref{5:phi-to-p}, and \eqref{quintic-lemma-prod} of Lemma~\ref{lemma:quintic}, we arrive at
\begin{align*}
\alpha\beta &= \frac{(uv)^5}{p^2} = \frac{1024q^5f^5(q, q^9)f^5(q^3, q^7)}{\varphi^{10}(q^5)p^2} = \frac{\{\varphi^2(q) - \varphi^2(q^5)\}^5}{\varphi^{10}(q^5)\Big(\frac{\varphi^2(q)}{\varphi^2(q^5)} - 1\Big)^2}\\
&= \frac{\{\varphi^2(q) - \varphi^2(q^5)\}^3}{\varphi^6(q^5)} = \bigg(\frac{\varphi^2(q)}{\varphi^2(q^5)} - 1\bigg)^3 = p^3.
\end{align*}
Thus, $\alpha$ and $\beta$ are roots of the equation in \eqref{5:r}.
\end{proof}

\section{Examples of quintic identities}

To establish quintic examples, we need the values of pairs of class invariants $G_n$ and $G_{25n}$, for certain rational numbers $n$. In Ramanujan's list \cite[pp.~189--199]{BerndtV} of values for $G_n$, there are $6$ values of $G_n$, namely for $n = 1, 3, 7, 9, 13,$ and $49$, for which $G_{25n}$ is also given. In analogy with the examples in Section~\ref{section:cubic-examples}, in view of \eqref{eq:G}, we have added the values for $n = 1/5$ to the present list. In this section, we determine the values of $15$ ratios of theta function. Together with the values given in Theorem~\ref{thm:thetagamma}, all of them can be expressed in terms of gamma functions. The examples given in this section are summarized in Table~\ref{table:quintic}. As in Section~\ref{section:cubic-examples}, there may be more than one way to determine the value of a theta function using ideas we have developed. Such cases are marked by an asterisk in Table~\ref{table:quintic}.
 The last two lines of the table build on the two class invariants $G_{125}$ and $G_{625}$, obtained in Corollaries~\ref{cor:G125} and~\ref{cor:G625}, which were not given by Ramanujan.

Throughout this section, we use the definitions of Theorem~\ref{thm:mainquintic}.

\begin{table}[ht]
\caption{Overview of quintic examples}\label{table:quintic}
\centering
\begin{tabular}{| r | r | r | l | r | l |}
	\hline &&&&&\\[-1em]
	$n$ & $25n$ & $5\sqrt{n}$ & Ex. for Thm.~\ref{thm:5n} & $25\sqrt{n}$ & Ex. for Thm.~\ref{thm:25n}\\
	\hline &&&&&\\[-1em]
	$1/5$ & $5$ & $\sqrt{5}$ & \eqref{eq:transform}\textsuperscript{$\ast$} & $5\sqrt{5}$ & Corollary~\ref{cor:5sqrt5} \\
	$1$ & $25$ & $5$ & Corollary~\ref{cor:5} & $25$ & Corollary~\ref{cor:25} \\
	$3$ & $75$ & $5\sqrt{3}$ & Corollary~\ref{cor:5sqrt3} & $25\sqrt{3}$ & Corollary~\ref{cor:25sqrt3} \\
	$7$ & $175$ & $5\sqrt{7}$ & Corollary~\ref{cor:5sqrt7} & $25\sqrt{7}$ & Corollary~\ref{cor:25sqrt7} \\
	$9$ & $225$ & $15$ & Corollary~\ref{cor:5:15} & $75$ & Corollary~\ref{cor:75} \\
	$13$ & $325$ & $5\sqrt{13}$ & Corollary~\ref{cor:5sqrt13} & $25\sqrt{13}$ & Corollary~\ref{cor:25sqrt13} \\
	$49$ & $1225$ & $35$ & Corollary~\ref{cor:35} & $175$ & Corollary~\ref{cor:175} \\
	\hline &&&&&\\[-1em]
	$5$ & $125$ & $5\sqrt{5}$ & Corollary~\ref{cor:5sqrt5}\textsuperscript{$\ast$} & $25\sqrt{5}$ & Corollary~\ref{cor:25sqrt5} \\
	$25$ & $625$ & $25$ & Corollary~\ref{cor:25}\textsuperscript{$\ast$} & $125$ & Corollary~\ref{cor:125} \\
	\hline
\end{tabular}
\end{table}

\begin{lemma}\label{lemma:5:p} If $q = e^{-\pi\sqrt{n}}$, for a positive rational number $n$, then
\begin{equation*}
p = \frac{2 G_n}{G_{25n}^5}.
\end{equation*}
\end{lemma}

\begin{proof}
From Theorem~\ref{thm:mainquintic}\eqref{5:p} and \eqref{def:G}, we have
\begin{equation*}
p = \frac{4q\chi(q)}{\chi^5(q^5)} = \frac{2 G_n}{G_{25n}^5}.\tag*{\qedhere}
\end{equation*}
\end{proof}

\begin{theorem}\label{thm:5n} If $n$ is a positive rational number, then
\begin{equation}\label{eq:5n}
\frac{\varphi(e^{-5\pi\sqrt{n}})}{\varphi(e^{-\pi\sqrt{n}})} = \frac{1}{\sqrt{5}}\bigg(1 + \frac{2 G_{25n}}{G_{n}^5}\bigg)^{1/2}.
\end{equation}
\end{theorem}

Theorem~\ref{thm:5n} is stated in \cite[p.~339, (8.11)]{BerndtV}, \cite[(1.20)]{BerndtChanZhang4}.

\begin{proof} By using Theorem~\ref{thm:mainquintic}\eqref{5:phi-to-p} with $q = e^{-\pi\sqrt{n}}$, and by Lemma~\ref{lemma:5:p}, we find that
\begin{equation*}
\frac{\varphi^2(e^{-\pi\sqrt{n}})}{\varphi^2(e^{-5\pi\sqrt{n}})} = 1 + \frac{2G_n}{G_{25n}^5}.
\end{equation*}
By the substitution $n \mapsto (25n)^{-1}$, \eqref{eq:G}, and two applications of the transformation formula \eqref{eq:transform}, after rearrangement, we finish the proof.
\end{proof}

\begin{theorem}\label{thm:25n} If $n$ is a positive rational number, then
\begin{gather}
\frac{\varphi(e^{-25\pi\sqrt{n}})}{\varphi(e^{-\pi\sqrt{n}})} = \frac{1}{5}\bigg(1 +\bigg\{\frac{G_{25n}}{G_{n}^5}\bigg[\bigg(\frac{2 G_{25n}}{G_{n}^5} - 1\bigg)^2 + 7 + \bigg(4 - \frac{2 G_{25n}}{G_{n}^5}\bigg)\bigg(4 + \frac{4 G_{25n}^2}{G_{n}^{10}}\bigg)^{1/2}\bigg]\bigg\}^{1/5} \notag\\
+ \bigg\{\frac{G_{25n}}{G_{n}^5}\bigg[\bigg(\frac{2 G_{25n}}{G_{n}^5} - 1\bigg)^2 + 7 - \bigg(4 - \frac{2 G_{25n}}{G_{n}^5}\bigg)\bigg(4 + \frac{4 G_{25n}^2}{G_{n}^{10}}\bigg)^{1/2}\bigg]\bigg\}^{1/5}\bigg).\label{eq:25n}
 \end{gather}
\end{theorem}

With the notation of Theorem~\ref{thm:mainquintic}, if $0 < q < 1$, then, from \cite[Lemma~3.2]{Rebak}, $u > v$. Thus, the first fifth root on the right-hand side of \eqref{eq:25n} is equal to $u$, and the second is equal to $v$.
\medskip

For
\begin{equation}\label{eq:s(p)-p}
p = \frac{2G_{25n}}{G_{n}^5},
\end{equation}
we define
\begin{align}\label{def:s(p)}
s(p) &:= \frac{1}{5}\Big(1 + \Big\{\frac{p}{2}\Big((p-1)^2 + 7 + (4 - p)(4 + p^2)^{1/2}\Big)\Big\}^{1/5}\notag \\
&\qquad\qquad\qquad + \Big\{\frac{p}{2}\Big((p-1)^2 + 7 - (4 - p)(4 + p^2)^{1/2}\Big)\Big\}^{1/5}\Big).
\end{align}
Then, we can record Theorem~\ref{thm:25n} in the form
\begin{equation}\label{eq:25n-s(p)}
\frac{\varphi(e^{-25\pi\sqrt{n}})}{\varphi(e^{-\pi\sqrt{n}})} = s(p).
\end{equation}

\begin{proof} First, we use the representation given in Theorem~\ref{thm:mainquintic}\eqref{5:1+u+v} with $q=e^{-\pi\sqrt{n}}$. Thus,
\begin{equation*}
\frac{\varphi(e^{-\pi\sqrt{n}/5})}{\varphi(e^{-5\pi\sqrt{n}})} = 1 +u +v.
\end{equation*}
Then, we solve the quadratic equation in Theorem~\ref{thm:mainquintic}\eqref{5:r}, where we recall that $p$ is given by Lemma~\ref{lemma:5:p}. Next, we use the expressions for $u$ and $v$ given in Theorem~\ref{thm:mainquintic}\eqref{5:u,v}. Lastly, we complete the proof by taking the same steps as in the end of the proof of Theorem~\ref{thm:5n}. By the substitution $n \mapsto (25n)^{-1}$, \eqref{eq:G}, and \eqref{eq:transform}, after rearrangement, we are done.
\end{proof}

\begin{corollary}\label{cor:5sqrt5}
We have
\begin{equation*}
\varphi(e^{-5\pi\sqrt{5}})
= 5^{-3/4}\varphi(e^{-\pi\sqrt{5}})\Bigg\{1 + \bigg(\frac{2(1 + \tan\frac{\pi}{5})}{1 - \sin\frac{\pi}{5}}\bigg)^{1/5} + \bigg(\frac{2(1 - \tan\frac{\pi}{5})}{1 + \sin\frac{\pi}{5}}\bigg)^{1/5}\Bigg\}.
\end{equation*}
\end{corollary}

Note that Corollary~\ref{cor:5sqrt5} is the quintic analogue of Theorem~\ref{thm:enigmatic}.

\begin{proof} We apply Theorem~\ref{thm:25n} in the form of \eqref{eq:25n-s(p)} with $n = 1/5$. Accordingly,
\begin{align*}
\frac{\varphi(e^{-5\sqrt5\pi})}{\varphi(e^{-\pi/\sqrt5})} = s(p),
\end{align*}
where by \eqref{def:s(p)}, and the transformation formula \eqref{eq:transform}, we see that
\begin{align}
    \varphi(e^{-5\sqrt5\pi}) =
    5^{-3/4}\varphi(e^{-\pi\sqrt5})
    &\Big(1 + \Big\{\frac{p}{2}\Big((p-1)^2 + 7 + (4 - p)(4 + p^2)^{1/2}\Big)\Big\}^{1/5}\notag \\
&\qquad+ \Big\{\frac{p}{2}\Big((p-1)^2 + 7 - (4 - p)(4 + p^2)^{1/2}\Big)\Big\}^{1/5}\Big),\label{eq:5sqrt5-transf}
\end{align}
and by \eqref{eq:s(p)-p}, with the value of $G_5$ given in \eqref{G5}, and by using \eqref{eq:G}, we find that
\begin{equation}\label{eq:5sqrt5-p}
p = \frac{2}{G_5^4} = \sqrt{5} - 1.
\end{equation}
Recall that
\begin{align}\label{eq:sin-cos-pi-5}
\sin\frac{\pi}{5} = \frac{\sqrt{10 - 2\sqrt{5}}}{4} \quad \text{and} \quad \cos\frac{\pi}{5} = \frac{\sqrt{5} + 1}{4}.
\end{align}
Now, employing \eqref{eq:5sqrt5-p} and \eqref{eq:sin-cos-pi-5}, we observe that
\begin{align}
\frac{1}{4}\Big((p-1)^2 + 7 \pm (4 - p)(4 + p^2)^{1/2}\Big) &= 4 - \sqrt{5} \pm \sqrt{25 - 10\sqrt{5}}\label{eq:5sqrt5-part-i-1}\\
&= \frac{\frac{\sqrt{10 - 2\sqrt{5}}}{4} \pm \frac{\sqrt{5} + 1}{4}}{1 \mp \frac{\sqrt{10 - 2\sqrt{5}}}{4}} = \frac{\cos\frac{\pi}{5} \pm \sin\frac{\pi}{5}}{1 \mp \sin\frac{\pi}{5}}\label{eq:5sqrt5-part-i-2}
\end{align}
and
\begin{equation}\label{eq:5sqrt5-part-ii}
2p = 2(\sqrt{5} - 1) = \frac{2}{\cos\frac{\pi}{5}}.
\end{equation}
Lastly, substitute \eqref{eq:5sqrt5-part-i-2} and \eqref{eq:5sqrt5-part-ii} into \eqref{eq:5sqrt5-transf} to complete the proof.
\end{proof}

We remark that by substituting the algebraic expressions in \eqref{eq:5sqrt5-part-i-1} and \eqref{eq:5sqrt5-part-ii} into \eqref{eq:5sqrt5-transf}, we find the alternative representation
\begin{multline*}
\varphi(e^{-5\pi\sqrt{5}}) = 5^{-3/4}\varphi(e^{-\pi\sqrt{5}})\Big(1 + \big\{2(\sqrt{5} - 1)(4 - \sqrt{5} + (25 - 10\sqrt{5})^{1/2})\big\}^{1/5}\\
+ \big\{2(\sqrt{5} - 1)(4 - \sqrt{5} - (25 - 10\sqrt{5})^{1/2})\big\}^{1/5}\Big).
\end{multline*}

\begin{corollary}\label{cor:5} We have
\begin{equation*}
\frac{\varphi(e^{-5\pi})}{\varphi(e^{-\pi})} = \frac{1}{(5\sqrt{5} - 10)^{1/2}}.
\end{equation*}
\end{corollary}

\begin{proof} We apply Theorem~\ref{thm:5n} with $n = 1$. We know that $G_1=1$ and, by \eqref{G25}, the value of $G_{25}$.
Thus, from \eqref{eq:5n},
\begin{equation*}
\frac{\varphi(e^{-5\pi})}{\varphi(e^{-\pi})} =\sqrt{\frac{1+2G_{25}}{5}} = \sqrt{\frac{2 + \sqrt{5}}{5}} =
\frac{1}{(5\sqrt{5} - 10)^{1/2}}.\tag*{\qedhere}
\end{equation*}
\end{proof}

Corollary~\ref{cor:5} can also be found in both the first \cite[p.~285]{RamanujanEarlierI} and second \cite[p.~104]{RamanujanEarlierII} notebooks of Ramanujan, and he also proposed a related problem \cite{Ramanujan629}, \cite[pp.~32--33]{BerndtChoiKang}. In addition to the two solutions of Ramanujan's problem given in \cite{Ramanujan629}, Corollary~\ref{cor:5} was later established by Heng Huat Chan and the first author \cite{BerndtChan}, \cite[p.~327]{BerndtV}.

We remark that since $\tan(\pi/5) = (5 - 2\sqrt{5})^{1/2}$, we have the trigonometric form
\begin{equation}\label{eq:5-trig}
\frac{\varphi(e^{-5\pi})}{\varphi(e^{-\pi})} = \frac{1}{5^{1/4}\tan\frac{\pi}{5}}.
\end{equation}

\begin{corollary}\label{cor:25} We have
\begin{equation*}
\frac{\varphi(e^{-25\pi})}{\varphi(e^{-\pi})} = \frac{1}{5}\Big\{1 + \big(8\big(3\cos\tfrac{\pi}{5} + \sin\tfrac{\pi}{5}\big)\big)^{1/5} +  \big(8\big(3\cos\tfrac{\pi}{5} - \sin\tfrac{\pi}{5}\big)\big)^{1/5}\Big\}.
\end{equation*}
\end{corollary}

\begin{proof} We apply Theorem~\ref{thm:25n} with $n = 1$. By \eqref{eq:s(p)-p}--\eqref{eq:25n-s(p)}, we arrive at
\begin{align}
\frac{\varphi(e^{-25\pi})}{\varphi(e^{-\pi})} &= \frac{1}{5}\Big(1 + \Big\{\frac{p}{2}\Big((p-1)^2 + 7 + (4 - p)(4 + p^2)^{1/2}\Big)\Big\}^{1/5}\notag\\
&\qquad\qquad\qquad + \Big\{\frac{p}{2}\Big((p-1)^2 + 7 - (4 - p)(4 + p^2)^{1/2}\Big)\Big\}^{1/5}\Big)\label{eq:25-s(p)},
\end{align}
where, using the value $G_1 = 1$ and the value of $G_{25}$ from \eqref{G25}, we find that
\begin{equation*}
p = 2G_{25} = \sqrt{5} + 1.
\end{equation*}
Now, recall the values of $\sin(\pi/5)$ and $\cos(\pi/5)$ from \eqref{eq:sin-cos-pi-5}, and observe that
\begin{align}
\frac{p}{2}\Big((p-1)^2 + 7 \pm (4 - p)(4 + p^2)^{1/2}\Big)
&= 2(1 + \sqrt{5})\bigg(3 \pm \sqrt{5 - 2\sqrt{5}}\bigg)\label{eq:25-1}\\
&= 6(1 + \sqrt{5}) \pm 2\sqrt{10 - 2\sqrt{5}}\notag\\
&= 8\big(3\cos\tfrac{\pi}{5} \pm \sin\tfrac{\pi}{5}\big)\label{eq:25-2}.
\end{align}
Lastly, inserting \eqref{eq:25-2} into \eqref{eq:25-s(p)}, we complete the proof.
\end{proof}

We note that if we substitute the algebraic expression in \eqref{eq:25-1} into \eqref{eq:25-s(p)}, we find the alternative formulation
\begin{equation*}
\frac{\varphi(e^{-25\pi})}{\varphi(e^{-\pi})} = \frac{1}{5}\Big(1 + \big\{2(1 + \sqrt{5})(3 + (5 - 2\sqrt{5})^{1/2})\big\}^{1/5} + \big\{2(1 + \sqrt{5})(3 - (5 - 2\sqrt{5})^{1/2})\big\}^{1/5}\Big).
\end{equation*}

\begin{corollary}\label{cor:5sqrt3} We have
\begin{equation*}
\frac{\varphi(e^{-5\pi\sqrt{3}})}{\varphi(e^{-\pi\sqrt{3}})} = \frac{1}{\sqrt{5}}\left(1 + p\right)^{1/2},
\end{equation*}
with
\begin{equation}\label{5sqrt3p}
p = \frac{6}{\dfrac{\sqrt{5} + 1}{2}(10)^{1/3} + \dfrac{\sqrt{5} - 1}{2}4^{1/3} \cdot 5^{1/6} - \sqrt{5} - 1}.
\end{equation}
\end{corollary}

\begin{proof} We apply Theorem~\ref{thm:5n} with $n = 3$. Recall the value $G_3$ from \eqref{G3}. Furthermore, from \cite[pp.~192, 269]{BerndtV},
\begin{equation*}
G_{75} = \frac{3 \cdot 2^{5/12}}{\dfrac{\sqrt{5} + 1}{2}(10)^{1/3} + \dfrac{\sqrt{5} - 1}{2}4^{1/3} \cdot 5^{1/6} - \sqrt{5} - 1}.
\end{equation*}
Using these values we obtain the desired result.
\end{proof}

\begin{corollary}\label{cor:25sqrt3} We have
\begin{equation*}
\frac{\varphi(e^{-25\pi\sqrt{3}})}{\varphi(e^{-\pi\sqrt{3}})} = s(p),
\end{equation*}
where $p$ is defined in \eqref{5sqrt3p} and $s(p)$ is defined by \eqref{def:s(p)}.
\end{corollary}

\begin{proof}
We apply Theorem~\ref{thm:25n} with $n = 3$. The proof is similar to the proof of Corollary~\ref{cor:5sqrt3}.
\end{proof}

\begin{corollary}\label{cor:5sqrt7} We have
\begin{equation}\label{eq:5sqrt7}
\frac{\varphi(e^{-5\pi\sqrt{7}})}{\varphi(e^{-\pi\sqrt{7}})} = \frac{1}{\sqrt{5}}(1 + p)^{1/2},
\end{equation}
with
\begin{equation}\label{5sqrt7p}
p = \frac{3}{\dfrac{\sqrt{5} - 1}{2} + \bigg(\dfrac{5 - \sqrt{5}}{4}\bigg)^{1/3}\Big((3\sqrt{21} + 8 - 3\sqrt{5})^{1/3} - (3\sqrt{21} - 8 + 3\sqrt{5})^{1/3}\Big)}.
\end{equation}
\end{corollary}

\begin{proof} We apply Theorem~\ref{thm:5n} with $n = 7$. We recall the value of $G_{7}$ from \eqref{G7} as well as the value of $G_{175}$ from \cite[pp.~195, 270]{BerndtV}, namely,
\begin{equation*}
G_{175} = \frac{3 \cdot 2^{1/4}}{\dfrac{\sqrt{5} - 1}{2} + \bigg(\dfrac{5 - \sqrt{5}}{4}\bigg)^{1/3}\Big((3\sqrt{21} + 8 - 3\sqrt{5})^{1/3} - (3\sqrt{21} - 8 + 3\sqrt{5})^{1/3}\Big)}.
\end{equation*}
Using these values we obtain the desired result.
\end{proof}

Using the representation for $G_{175}$ given in \cite[p.~272, (8.13)]{BerndtV}, we can obtain another formula for \eqref{eq:5sqrt7}, with
\begin{align}
p = \frac{1}{3}\bigg(2 + \sqrt{5} + \bigg(\frac{5 + 2\sqrt{5}}{2}\bigg)^{1/3}\Big((17 + 3\sqrt{21})^{1/3} + (17 - 3\sqrt{21})^{1/3}\Big)\bigg)\label{5sqrt7palt}.
\end{align}

\begin{corollary}\label{cor:25sqrt7} We have
\begin{equation*}
\frac{\varphi(e^{-25\pi\sqrt{7}})}{\varphi(e^{-\pi\sqrt{7}})} = s(p),
\end{equation*}
where $p$ is defined in \eqref{5sqrt7p} or \eqref{5sqrt7palt}, and $s(p)$ is defined by \eqref{def:s(p)}.
\end{corollary}

\begin{proof}
We apply Theorem~\ref{thm:25n} with $n = 7$. The proof is similar to the proof of Corollary~\ref{cor:5sqrt7}.
\end{proof}

Next, we give another representation for $\varphi(e^{-15\pi})/\varphi(e^{-\pi})$, given in Corollary~\ref{cor:3:15}.

\begin{corollary}\label{cor:5:15} We have
\begin{equation*}
\frac{\varphi(e^{-15\pi})}{\varphi(e^{-\pi})} = \frac{1}{\sqrt{5}(6\sqrt{3} - 9)^{1/4}}\bigg(1 + 2(2 - \sqrt{3})^{1/2}\bigg(\frac{1 + \sqrt{5}}{2}\bigg)\bigg(\frac{\sqrt{4 + \sqrt{15}} + (15)^{1/4}}{2}\bigg)\bigg)^{1/2}.
\end{equation*}
\end{corollary}

\begin{proof} Apply Theorem~\ref{thm:5n} with $n = 9$. Hence,
\begin{equation}\label{ppp}
\frac{\varphi(e^{-15\pi})}{\varphi(e^{-3\pi})} = \frac{1}{\sqrt{5}}\bigg(1 + \frac{2 G_{225}}{G_{9}^5}\bigg)^{1/2}.
\end{equation}
Recall the value \eqref{G9} of $G_9$, and observe that
\begin{equation}\label{G9a}
G_9 = \bigg(\frac{1 + \sqrt{3}}{\sqrt{2}}\bigg)^{1/3} = \bigg(\frac{\sqrt{3} - 1}{\sqrt{2}}\bigg)^{-1/3}.
\end{equation}
Also, from \cite[p.~195]{BerndtV},
\begin{equation}\label{G225}
G_{225} = \bigg(\frac{1 + \sqrt{5}}{2}\bigg)(2 + \sqrt{3})^{1/3}\bigg(\frac{\sqrt{4 + \sqrt{15}} + (15)^{1/4}}{2}\bigg).
\end{equation}
Note that
\begin{equation}\label{G9G225}
\bigg(\frac{\sqrt{3} - 1}{\sqrt{2}}\bigg)^{5/3}(2 + \sqrt{3})^{1/3} = \sqrt{2 - \sqrt{3}}.
\end{equation}
Put \eqref{G9a} and \eqref{G225} into \eqref{ppp} and simplify with the aid of \eqref{G9G225}. Lastly, multiply both sides of \eqref{ppp} by \eqref{e3}.
The proof of Corollary~\ref{cor:5:15} is then complete.
\end{proof}

\begin{corollary}\label{cor:75} We have
\begin{equation*}
\frac{\varphi(e^{-75\pi})}{\varphi(e^{-\pi})} = \frac{s(p)}{(6\sqrt{3} - 9)^{1/4}},
\end{equation*}
where
\begin{equation*}
p = 2(2 - \sqrt{3})^{1/2}\bigg(\frac{1 + \sqrt{5}}{2}\bigg)\bigg(\frac{\sqrt{4 + \sqrt{15}} + (15)^{1/4}}{2}\bigg),
\end{equation*}
and $s(p)$ is defined by \eqref{def:s(p)}.
\end{corollary}

\begin{proof} We apply Theorem~\ref{thm:25n} for $n = 9$. The remainder of the proof follows along the same lines as the proof of Corollary~\ref{cor:5:15}.
\end{proof}

\begin{corollary}\label{cor:5sqrt13} We have
\begin{equation*}
\frac{\varphi(e^{-5\pi\sqrt{13}})}{\varphi(e^{-\pi\sqrt{13}})} = \frac{1}{\sqrt{5}}\bigg(1 + (\sqrt{13} - 3)t\bigg)^{1/2},
\end{equation*}
where
\begin{multline}\label{5sqrt13t}
t^3 + t^2\bigg(\frac{1 - \sqrt{13}}{2}\bigg)^{2} + t\bigg(\frac{1 + \sqrt{13}}{2}\bigg)^{2} + 1 =\\
\sqrt{5}\bigg\{t^3 - t^2\bigg(\frac{1 + \sqrt{13}}{2}\bigg) + t\bigg(\frac{1 - \sqrt{13}}{2}\bigg) - 1\bigg\}.
\end{multline}
\end{corollary}

\begin{proof} We apply Theorem~\ref{thm:5n} with $n = 13$. Recall the value of $G_{13} $ from \eqref{G13}. Also, from \cite[p.~196]{BerndtV},
\begin{equation*}
G_{325} = \bigg(\frac{3 + \sqrt{13}}{2}\bigg)^{1/4}\,t,
\end{equation*}
where $t$ satisfies \eqref{5sqrt13t}. Using these values we obtain the desired result.
\end{proof}

\begin{corollary}\label{cor:25sqrt13} We have
\begin{equation*}
\frac{\varphi(e^{-25\pi\sqrt{13}})}{\varphi(e^{-\pi\sqrt{13}})} = s(p),
\end{equation*}
with
\begin{equation*}
p = (\sqrt{13} - 3)t,
\end{equation*}
where $t$ satisfies \eqref{5sqrt13t}, and $s(p)$ is defined by \eqref{def:s(p)}.
\end{corollary}

\begin{proof}
We apply Theorem~\ref{thm:25n} with $n = 13$. The proof is similar to the proof of Corollary~\ref{cor:5sqrt13}.
\end{proof}

\begin{corollary}\label{cor:35} We have
\begin{equation*}
\frac{\varphi(e^{-35\pi})}{\varphi(e^{-\pi})} = \frac{1}{\sqrt{5}}\bigg(\frac{\sqrt{13 + \sqrt{7}} + \sqrt{7 + 3\sqrt{7}}}{14}(28)^{1/8}\bigg)^{1/2}(1+p)^{1/2},
\end{equation*}
where
\begin{align}
&\hspace{-0.3em}p = 2\bigg(\frac{1 + \sqrt{5}}{2}\bigg)(6 + \sqrt{35})^{1/4}\Bigg(\frac{\sqrt{4 + \sqrt{7}} - 7^{1/4}}{2}\Bigg)^{7/2}\notag\\
&\hspace{-0.3em}\times \left(\sqrt{\frac{43 + 15\sqrt{7} + (8 + 3\sqrt{7})\sqrt{10\sqrt{7}}}{8}} + \sqrt{\frac{35 + 15\sqrt{7} + (8 + 3\sqrt{7})\sqrt{10\sqrt{7}}}{8}}\right).\label{35}
\end{align}
\end{corollary}

Another representation for $\varphi(e^{-35\pi})/\varphi(e^{-\pi})$ is given in \cite[Theorem~6.5]{Rebak}.

\begin{proof} We apply Theorem~\ref{thm:5n} with $n = 49$. Recall the value of $G_{49} $ from \eqref{G49}. Also, from \cite[p.~199]{BerndtV},
\begin{multline*}
G_{1225} = \frac{1+\sqrt{5}}{2}(6+\sqrt{35})^{1/4}\bigg(\frac{7^{1/4} + \sqrt{4+\sqrt{7}}}{2}\bigg)^{3/2}\\
\times \left(\sqrt{\frac{43 + 15\sqrt{7} + (8 + 3\sqrt{7})\sqrt{10\sqrt{7}}}{8}} + \sqrt{\frac{35 + 15\sqrt{7} + (8 + 3\sqrt{7})\sqrt{10\sqrt{7}}}{8}}\right).
\end{multline*}
Substitute these values in \eqref{eq:25n}. If we combine the resulting identity with \eqref{e7}, we complete the proof.
\end{proof}

\begin{corollary}\label{cor:175} We have
\begin{equation*}
\frac{\varphi(e^{-175\pi})}{\varphi(e^{-\pi})} = \bigg(\frac{\sqrt{13 + \sqrt{7}} + \sqrt{7 + 3\sqrt{7}}}{14}(28)^{1/8}\bigg)^{1/2}s(p),
\end{equation*}
where $p$ is provided by \eqref{35}, and $s(p)$ is given by \eqref{def:s(p)}.
\end{corollary}

\begin{proof} Apply Theorem~\ref{thm:25n} with $n = 49$. Now proceed as in the proof of Corollary~\ref{cor:35}.
\end{proof}

\begin{lemma}\label{lemma:G25n} If $n$ is a positive rational number, then
\begin{equation*}
G_{25n} = \frac{G_n^5}{2}\bigg(5\frac{\varphi^2(e^{-5\pi\sqrt{n}})}{\varphi^2(e^{-\pi\sqrt{n}})} - 1\bigg).
\end{equation*}
\end{lemma}

\begin{proof}
This is a direct corollary of Theorem~\ref{thm:5n}.
\end{proof}

\begin{corollary}\label{cor:G125} We have
\begin{multline*}
G_{125} = \frac12\bigg(\frac{11 + 5\sqrt{5}}{2}\bigg)^{1/4}\bigg(\frac{1}{\sqrt{5}}\bigg\{1 + \bigg(\frac{2(1 + \tan\frac{\pi}{5})}{1 - \sin\frac{\pi}{5}}\bigg)^{1/5} + \bigg(\frac{2(1 - \tan\frac{\pi}{5})}{1 + \sin\frac{\pi}{5}}\bigg)^{1/5}\bigg\}^2 - 1\bigg).
\end{multline*}
\end{corollary}

\begin{proof} We use Lemma~\ref{lemma:G25n} with $n = 5$. Recall the value of $G_5$ from \eqref{G5}. Since from Corollary~\ref{cor:5sqrt5} we know the value of $\varphi(e^{-5\pi\sqrt{5}})/\varphi(e^{-\pi\sqrt{5}})$, after simplification, we obtain the given form.
\end{proof}

\begin{corollary}\label{cor:G625} We have
\begin{multline*}
G_{625} = \frac{11 + 5\sqrt{5}}{4}\bigg(\frac{\tan^2 \frac{\pi}{5}}{\sqrt{5}}\Big\{1 + \big(8\big(3\cos\tfrac{\pi}{5} + \sin\tfrac{\pi}{5}\big)\big)^{1/5} +  \big(8\big(3\cos\tfrac{\pi}{5} - \sin\tfrac{\pi}{5}\big)\big)^{1/5}\Big\}^2 - 1\bigg).
\end{multline*}
\end{corollary}

\begin{proof} Apply Lemma~\ref{lemma:G25n} with $n = 25$. Recall the value of $G_{25}$ from \eqref{G25}. The remainder of the proof is straightforward. We need the values of
\begin{equation*}
 \frac{\varphi(e^{-\pi})}{\varphi(e^{-5\pi})} \qquad \text{and} \qquad
\frac{\varphi(e^{-25\pi})}{\varphi(e^{-\pi})}
\end{equation*}
from \eqref{eq:5-trig} and Corollary~\ref{cor:25}, respectively. After simplification, the corollary follows.
\end{proof}

Corollaries~\ref{cor:G125} and~\ref{cor:G625} were not given by Ramanujan in his extensive compendium of class invariants \cite[pp.~189--199]{BerndtV}. Furthermore, we remark that
\begin{equation*}
\frac{11 + 5\sqrt{5}}{4} = 16\cos^5\tfrac{\pi}{5}.
\end{equation*}

\begin{corollary}\label{cor:25sqrt5} We have
\begin{equation*}
\frac{\varphi(e^{-25\pi\sqrt{5}})}{\varphi(e^{-\pi\sqrt{5}})} = s(p),
\end{equation*}
where
\begin{equation*}
p = \frac{1}{\sqrt{5}}\bigg\{1 + \bigg(\frac{2(1 + \tan\frac{\pi}{5})}{1 - \sin\frac{\pi}{5}}\bigg)^{1/5} + \bigg(\frac{2(1 - \tan\frac{\pi}{5})}{1 + \sin\frac{\pi}{5}}\bigg)^{1/5}\bigg\}^2 - 1,
\end{equation*}
and $s(p)$ is given by \eqref{def:s(p)}.
\end{corollary}

\begin{proof} Apply Theorem~\ref{thm:25n} with $n = 5$, and also use \eqref{def:s(p)}. According to \eqref{eq:s(p)-p}, to calculate the value of
\begin{equation*}
p = \frac{2G_{25n}}{G_{n}^5},
\end{equation*}
we need the value of $G_{125}$ given in Corollary~\ref{cor:G125}, and the value of $G_5$ provided in \eqref{G5}.
Note that
\begin{equation*}
\bigg(\frac{11 + 5\sqrt{5}}{2}\bigg)^{1/4}\bigg(\frac{1 + \sqrt{5}}{2}\bigg)^{-5/4} = 1.
\end{equation*}
The remainder of the proof is straightforward.
\end{proof}

\begin{corollary}\label{cor:125} We have
\begin{equation*}
\frac{\varphi(e^{-125\pi})}{\varphi(e^{-\pi})} = \frac{s(p)}{(5\sqrt{5} - 10)^{1/2}},
\end{equation*}
where
\begin{equation*}
p = \frac{\tan^2 \frac{\pi}{5}}{\sqrt{5}}\Big\{1 + \big(8\big(3\cos\tfrac{\pi}{5} + \sin\tfrac{\pi}{5}\big)\big)^{1/5} +  \big(8\big(3\cos\tfrac{\pi}{5} - \sin\tfrac{\pi}{5}\big)\big)^{1/5}\Big\}^2 - 1,
\end{equation*}
and $s(p)$ is given by \eqref{def:s(p)}.
\end{corollary}

\begin{proof} Apply Theorem~\ref{thm:25n} with $n = 25$, and also use \eqref{def:s(p)}. To obtain the value of $p$ in \eqref{eq:s(p)-p}, we need the value of $G_{625}$ given in Corollary~\ref{cor:G625} and the value of $G_{25}$ provided in \eqref{G25}.
Note that
\begin{equation*}
2\bigg(\frac{11 + 5\sqrt{5}}{4}\bigg)\bigg(\frac{1 + \sqrt{5}}{2}\bigg)^{-5} = 1.
\end{equation*}
The value for $\varphi(e^{-5\pi})/\varphi(e^{-\pi})$ is given in Corollary~\ref{cor:5}. The proof is now straightforwardly completed.
\end{proof}

Corollaries~\ref{cor:G125}--\ref{cor:125} can be expressed by radicals using \eqref{eq:sin-cos-pi-5}.

\section{Values of the Borweins' cubic theta function $a(q)$}

Recall that the Borweins' cubic theta function $a(q)$ is defined by \eqref{def:a}. This section is devoted to determining specific values of $a(q)$.
We first provide the following theorem, which is analogous to Theorems~\ref{thm:3n} and~\ref{thm:5n}.

\begin{theorem}\label{thm:a} If $n$ is a positive rational number, then
\begin{equation*}
\frac{a(e^{-2\pi\sqrt{n}})}{\varphi^2(e^{-\pi\sqrt{n}})} = \frac{1}{3}\bigg(1 + \frac{2\sqrt{2}G_n^3}{G_{9n}^9}\bigg)^{1/4}\bigg(1 + \frac{\sqrt{2}G_{9n}^3}{2G_n^9}\bigg).
\end{equation*}
\end{theorem}

\begin{proof}
From \cite[(6.4)]{BerndtBhargavaGarvan},
\begin{equation*}
\frac{a(q^2)}{\varphi^2(q)} = \frac{1}{4}\frac{\varphi(q)}{\varphi(q^3)} + \frac{3}{4}\frac{\varphi^3(q^3)}{\varphi^3(q)} = \frac{1}{4}\frac{\varphi(q)}{\varphi(q^3)}\bigg(1 + 3\frac{\varphi^4(q^3)}{\varphi^4(q)}\bigg).
\end{equation*}
Set $q := e^{-\pi\sqrt{n}}$. Using \eqref{eq:3sqrtn} and Theorem~\ref{thm:3n}, after rearrangement, we complete the proof of Theorem~\ref{thm:a}.
\end{proof}

\begin{corollary} We have
\begin{equation*}
\frac{a(e^{-2\pi/\sqrt{3}})}{\varphi^2(e^{-\pi\sqrt{3}})} = \frac{3^{3/4}}{2}.
\end{equation*}
\end{corollary}

\begin{proof} We apply Theorem~\ref{thm:a} with $n = 1/3$. Using the value of $G_3$ from \eqref{G3} and with the use of \eqref{eq:G}, we find that
\begin{equation*}
\frac{a(e^{-2\pi/\sqrt{3}})}{\varphi^2(e^{-\pi/\sqrt{3}})} = \frac{1}{3}\bigg(1 + \frac{2\sqrt{2}\cdot 2^{1/4}}{2^{3/4}}\bigg)^{1/4}\bigg(1 + \frac{\sqrt{2}\cdot 2^{1/4}}{2\cdot 2^{3/4}}\bigg) = \frac{3^{1/4}}{2}.
\end{equation*}
Using the transformation formula \eqref{eq:transform}, we complete the proof.
\end{proof}

\begin{corollary} We have
\begin{equation*}
\frac{a(e^{-2\pi})}{\varphi^2(e^{-\pi})} = \bigg(\frac{1}{4} + \frac{1}{2\sqrt{3}}\bigg)^{1/4}.
\end{equation*}
\end{corollary}

\begin{proof}
We apply Theorem~\ref{thm:a} with $n = 1$. Employ the value $G_1 = 1$ and the value of $G_9$ given in \eqref{G9}. Thus,
\begin{align*}
\frac{a(e^{-2\pi})}{\varphi^2(e^{-\pi})} &= \frac{1}{3}\Bigg(1 + 2\sqrt{2}\bigg(\frac{\sqrt{3} - 1}{\sqrt{2}}\bigg)^3\Bigg)^{1/4}\bigg(1 + \frac{\sqrt{2}}{2}\bigg(\frac{\sqrt{3} + 1}{\sqrt{2}}\bigg)\bigg)\\
&= (3(2\sqrt{3} - 3))^{1/4}\bigg(\frac{3 + \sqrt{3}}{6}\bigg)\\
&= \bigg(\frac{(2\sqrt{3} - 3)(7 + 4\sqrt{3})}{12}\bigg)^{1/4}\\
&= \bigg(\frac{3 + 2\sqrt{3}}{12}\bigg)^{1/4}\\
&= \bigg(\frac{1}{4} + \frac{1}{2\sqrt{3}}\bigg)^{1/4}.\tag*{\qedhere}
\end{align*}
\end{proof}

The first author and Heng Huat Chan~\cite{BerndtChan}, \cite[p.~328]{BerndtV} had previously obtained the equivalent value
\begin{equation*}
\frac{a(e^{-2\pi})}{\varphi^2(e^{-\pi})} = \frac{1}{(12)^{1/8} (\sqrt{3} - 1)^{1/2}}.
\end{equation*}

\begin{corollary} We have
\begin{equation*}
\frac{a(e^{-2\pi\sqrt{3}})}{\varphi^2(e^{-\pi\sqrt{3}})} = \frac{3^{3/4}(1 + 2^{2/3})}{6}.
\end{equation*}
\end{corollary}

\begin{proof}
We apply Theorem~\ref{thm:a} with $n = 3$. Utilizing the values of $G_3$ and $G_{27}$ given by \eqref{G3} and \eqref{G27}, respectively, we deduce that
\begin{align*}
\frac{a(e^{-2\pi\sqrt{3}})}{\varphi^2(e^{-\pi\sqrt{3}})} &= \frac{1}{3}\Bigg(1 + 2\sqrt{2}\bigg(\frac{2^{1/3} - 1}{2^{1/4}}\bigg)^3 2^{1/4}\Bigg)^{1/4}\bigg(1 + \frac{\sqrt{2}}{2}2^{-3/4}\bigg(\frac{2^{1/4}}{2^{1/3} - 1}\bigg)\bigg)\\
&= \frac{1}{3}(3(1 + 2^{4/3} - 2^{5/3}))^{1/4}\bigg(\frac{2^{4/3} - 1}{2^{4/3} - 2}\bigg)\\
&= \frac{1}{3}(3(2^{2/3} - 1)^2)^{1/4}\bigg(\frac{2^{4/3} - 1}{2^{4/3} - 2}\bigg)\\
&= \frac{(2^{2/3} - 1)^{1/2}}{3^{3/4}}\bigg(\frac{(2^{2/3} + 1)(2^{2/3} - 1)}{2^{4/3} - 2}\bigg)\\
&= \frac{(2^{2/3} - 1)^{1/2}(2^{2/3} + 1)}{3^{3/4}}\bigg(\frac{2^{1/3} + 1}{2}\bigg)\\
&= \frac{(2^{2/3} - 1)^{1/2}(2^{1/3} + 1)}{3^{3/4}}\bigg(\frac{2^{2/3} + 1}{2}\bigg)\\
&= \frac{\sqrt{3}}{3^{3/4}}\bigg(\frac{2^{2/3} + 1}{2}\bigg)\\
&= \frac{3^{3/4}(1 + 2^{2/3})}{6}.\tag*{\qedhere}
\end{align*}
\end{proof}

\begin{corollary} We have
\begin{equation*}
\frac{a(e^{-2\pi\sqrt{5}})}{\varphi^2(e^{-\pi\sqrt{5}})} = \frac{1}{3}\left(1 + 2\sqrt{2}(\sqrt{5} - 2)^2\bigg(\frac{\sqrt{5} - \sqrt{3}}{\sqrt{2}}\bigg)^3\right)^{1/4}\bigg(1 + \frac{\sqrt{2}}{2}\bigg(\frac{\sqrt{5} + \sqrt{3}}{\sqrt{2}}\bigg)\bigg).
\end{equation*}
\end{corollary}

\begin{proof}
We apply Theorem~\ref{thm:a} with $n = 5$. The proof is similar to that for Corollary~\ref{cor:3sqrt5}.
\end{proof}

\begin{corollary} We have
\begin{align*}
\frac{a(e^{-2\pi\sqrt{7}})}{\varphi^2(e^{-\pi\sqrt{7}})} &= \frac{1}{3}\left(1 + \bigg(\frac{\sqrt{7} - \sqrt{3}}{2}\bigg)^3\left(\sqrt{\frac{5 + \sqrt{21}}{8}} - \sqrt{\frac{\sqrt{21} - 3}{8}}\right)^9\right)^{1/4}\\
&\qquad\qquad\times \left(1 + \frac{1}{4}\bigg(\frac{\sqrt{7} + \sqrt{3}}{2}\bigg)\left(\sqrt{\frac{5 + \sqrt{21}}{8}} + \sqrt{\frac{\sqrt{21} - 3}{8}}\right)^3\right).
\end{align*}
\end{corollary}

\begin{proof}
We apply Theorem~\ref{thm:a} with $n = 7$. The proof is similar to that for Corollary~\ref{cor:3sqrt7}.
\end{proof}

\begin{corollary} We have
\begin{align*}
\frac{a(e^{-6\pi})}{\varphi^2(e^{-\pi})} = \frac{1}{3(6\sqrt{3} - 9)^{1/2}}&\left(1 + 2\sqrt{2}\left(\frac{(2(\sqrt{3} - 1))^{1/3} - 1}{(2(\sqrt{3} + 1))^{1/3} + 1}\right)^3 \bigg(\frac{\sqrt{3} + 1}{\sqrt{2}}\bigg)\right)^{1/4}\\
&\qquad\times\left(1 + \frac{\sqrt{2}}{2}\bigg(\frac{\sqrt{3} - 1}{\sqrt{2}}\bigg)^3\left(\frac{(2(\sqrt{3} + 1))^{1/3} + 1}{(2(\sqrt{3} - 1))^{1/3} - 1}\right)\right).
\end{align*}
\end{corollary}

\begin{proof}
We apply Theorem~\ref{thm:a} when $n = 9$. Employing the values of $G_9$ and $G_{81}$ given in \eqref{G9} and \eqref{G81}, respectively, and combining them with Corollary~\ref{cor:3}, we complete the proof.
\end{proof}

\begin{corollary} We have
\begin{align*}
&\frac{a(e^{-10\pi})}{\varphi^2(e^{-\pi})} = \frac{1}{3(5\sqrt{5} - 10)}\\
&\,\times\left(1 + 2\sqrt{2}(\sqrt{5} - 2)^2(2 - \sqrt{3})^3\bigg(\sqrt{4276 + 1104\sqrt{15}} - \sqrt{4275 + 1104\sqrt{15}}\bigg)^{3/2}\right)^{1/4}\\
&\,\times\left(1 + \frac{\sqrt{2}}{2}(\sqrt{5} - 2)^2(2 + \sqrt{3})\bigg(\sqrt{4276 + 1104\sqrt{15}} + \sqrt{4275 + 1104\sqrt{15}}\bigg)^{1/2}\right).
\end{align*}
\end{corollary}

\begin{proof}
Invoke Theorem~\ref{thm:a} with $n = 25$. The proof is similar to that for Corollary~\ref{cor:3:15}.
\end{proof}

\begin{corollary} We have
\begin{align*}
&\frac{a(e^{-14\pi})}{\varphi^2(e^{-\pi})} = \frac{1}{3}\left(\frac{\sqrt{13 + \sqrt{7}} + \sqrt{7 + 3\sqrt{7}}}{14} (28)^{1/8}\right)\\
&\times \left(1 + 2\sqrt{2}(2\sqrt{7} - 3\sqrt{3})^{3/2}(2 - \sqrt{3})^{3/2}\vphantom{\left(\frac{\sqrt{3 + \sqrt{7}} - (6\sqrt{7})^{1/4}}{\sqrt{3 + \sqrt{7}} + (6\sqrt{7})^{1/4}}\right)^{9/2}}\right.\\
&\qquad\qquad \times \left.\left(\frac{\sqrt{3 + \sqrt{7}} - (6\sqrt{7})^{1/4}}{\sqrt{3 + \sqrt{7}} + (6\sqrt{7})^{1/4}}\right)^{9/2}\left(\sqrt{932 + 352\sqrt{7}} - \sqrt{931 + 352\sqrt{7}}\right)\right)^{1/4}\\
&\times \left(1 + \frac{\sqrt{2}}{2}\left(\sqrt{932 + 352\sqrt{7}} - \sqrt{931 + 352\sqrt{7}}\right)\vphantom{\left(\frac{\sqrt{3 + \sqrt{7}} + (6\sqrt{7})^{1/4}}{\sqrt{3 + \sqrt{7}} - (6\sqrt{7})^{1/4}}\right)^{3/2}}\right.\\
&\qquad\qquad \times \left.(2\sqrt{7} + 3\sqrt{3})^{1/2}(2 + \sqrt{3})^{1/2}\left(\frac{\sqrt{3 + \sqrt{7}} + (6\sqrt{7})^{1/4}}{\sqrt{3 + \sqrt{7}} - (6\sqrt{7})^{1/4}}\right)^{3/2}\right).
\end{align*}
\end{corollary}

\begin{proof}
Apply Theorem~\ref{thm:a} with $n = 49$. The proof is similar to that for Corollary~\ref{cor:21}.
\end{proof}

\begin{corollary} We have
\begin{align*}
&\frac{a(e^{-18\pi})}{\varphi^2(e^{-\pi})} = \frac{(1 + (2(\sqrt{3} + 1))^{1/3})^2}{27} \\
&\,\times\left(1 + \left(\frac{3(\sqrt{3} + 1)}{\sqrt{3} + 1 + 2\left(\frac{(2(\sqrt{3} + 1))^{1/3} + 1}{(2(\sqrt{3} - 1))^{1/3} - 1}\right)^{1/3}} - 1\right)^3\right)^{1/4}\\
&\,\times\left(1 + \frac{1}{2}\left(\frac{(2(\sqrt{3} - 1))^{1/3} - 1}{(2(\sqrt{3} + 1))^{1/3} + 1}\right)^{8/3}\left(\frac{3(\sqrt{3} + 1)}{\sqrt{3} + 1 - \left(\frac{(2(\sqrt{3} + 1))^{1/3} + 1}{(2(\sqrt{3} - 1))^{1/3} - 1}\right)^{1/3}} - 2\right)\right).
\end{align*}
\end{corollary}

\begin{proof}
Invoke Theorem~\ref{thm:a} with $n = 81$. The proof is similar to the proof of Corollary~\ref{cor:81}. After simplification, we obtain the given form.
\end{proof}

Further examples of Theorem~\ref{thm:a} can be given for $n = 11, 13, 17, 27, 37, $ and $85$.

We have calculated cubic and quintic examples that are reachable via Ramanujan's theory. The second author obtained examples for the septic case in \cite{Rebak}. As we mentioned in the introduction, these are special cases of a hypothetical theorem that Ramanujan mentioned in a note at the end of Section~12 of Chapter~20 in his second notebook \cite[p.~247]{RamanujanEarlierII}, \cite[p.~400]{BerndtIII}. A possible further direction of research could be the examination of higher-order identities and accessible examples for them. However, we must emphasize that as the order increases, we know less and less about the modular equations necessary for a possible grand theory.

\subsubsection*{Acknowledgments} Theorem~\ref{thm:e37} appeared as a conjecture in the first version of our paper on arXiv.  Subsequently, Dongxi Ye proved our conjecture.  We are grateful to him for allowing us to present a slightly modified version of his proof. We also thank the referee for helpful comments, in particular, for drawing our attention to the value of $G_{99}$, which was needed to prove Corollaries~\ref{cor:3sqrt11} and~\ref{cor:9sqrt11}.


\begin{thebibliography}{10}

\bibitem{AndrewsBerndtII}
G.~E. Andrews and B.~C. Berndt.
\newblock {\em {Ramanujan's Lost Notebook, Part II.}}
\newblock Springer, New York, 2009.

\bibitem{BerndtIII}
B.~C. Berndt.
\newblock {\em {Ramanujan's Notebooks, Part III.}}
\newblock Springer-Verlag, New York, 1991.

\bibitem{BerndtV}
B.~C. Berndt.
\newblock {\em {Ramanujan's Notebooks, Part V.}}
\newblock Springer-Verlag, New York, 1998.

\bibitem{BerndtBhargavaGarvan}
B.~C. Berndt, S.~Bhargava, and F.~G. Garvan.
\newblock {Ramanujan's theories of elliptic functions to alternative bases}.
\newblock {\em Trans. Amer. Math. Soc.}, 347(11):4163--4244, 1995.

\bibitem{BerndtChan}
B.~C. Berndt and H.~H. Chan.
\newblock {Ramanujan's explicit values for the classical theta function}.
\newblock {\em Mathematika}, 42(2):278--294, 1995.

\bibitem{BerndtChanZhang4}
B.~C. Berndt, H.~H. Chan, and L.-C. Zhang.
\newblock {Ramanujan's remarkable product of theta-functions}.
\newblock {\em Proc. Edinburgh Math. Soc.}, 40(3):583--612, 1997.

\bibitem{BerndtChoiKang}
B.~C. Berndt, Y.-S. Choi, and S.-Y. Kang.
\newblock {The problems submitted by Ramanujan to the Journal of the Indian
  Mathematical Society}.
\newblock In {\em Continued Fractions: From Analytic Number Theory to
  Constructive Approximation, \emph{B.~C.~Berndt and F.~Gesztesy, eds.,
  Contemp.~Math.~236, Amer.~ Math.~ Soc., Providence, RI, 1999, pp.~15--56;
  reprinted in: B.~C.~Berndt and R.~A.~Rankin, eds., \emph{Ramanujan: Essays
  and Surveys}, jointly published by Amer.~Math.~Soc.~and London
  Math.~Soc.~2001, pp.~ 215--258.}}

\bibitem{BerndtRebak}
B.~C. Berndt and {\"{O}}.~Reb{\'{a}}k.
\newblock {Explicit values for Ramanujan's theta function $\varphi(q)$}.
\newblock {\em Hardy--Ramanujan J.}, 44(1):41--50, 2021.

\bibitem{BorweinBrothersPiAGM}
J.~M. Borwein and P.~B. Borwein.
\newblock {\em {Pi and the AGM}}.
\newblock Wiley, New York, 1987.

\bibitem{BorweinBrothersApproxPi}
J.~M. Borwein and P.~B. Borwein.
\newblock {Approximating $\pi$ with Ramanujan's modular equations}.
\newblock {\em Rocky Mountain J. Math.}, 19(1):93--102, 1989.

\bibitem{BorweinBrothersCubic}
J.~M. Borwein and P.~B. Borwein.
\newblock {A cubic counterpart of Jacobi's identity and the AGM}.
\newblock {\em Trans. Amer. Math. Soc.}, 323(2):691--701, 1991.

\bibitem{BorweinZucker}
J.~M. Borwein and I.~J. Zucker.
\newblock {Fast evaluation of the gamma function for small rational fractions
  using complete elliptic integrals of the first kind}.
\newblock {\em IMA J. Numer. Anal.}, 12(4):519--526, 1992.

\bibitem{Cox}
D.~A. Cox.
\newblock {\em {Primes of the Form $x^2 + ny^2$}}.
\newblock Wiley, New York, 1989.

\bibitem{Dickson}
L.~E. Dickson.
\newblock {\em {Introduction to the Theory of Numbers}}.
\newblock Dover, New York, 1957.

\bibitem{HuardKaplanWilliams}
J.~G. Huard, P.~Kaplan, and K.~S. Williams.
\newblock {The Chowla--Selberg formula for genera}.
\newblock {\em Acta Arithmetica}, 73(3):271--301, 1995.

\bibitem{Jacobi}
C.~G.~J. Jacobi.
\newblock {\em {Fundamenta nova theoriae functionum ellipticarum}}.
\newblock Sumptibus fratrum Borntr\ae ger, {Regiomonti}, 1829.

\bibitem{Knopp}
M.~I. Knopp.
\newblock {\em {Modular Functions in Analytic Number Theory}}.
\newblock Markham, Chicago, 1970.

\bibitem{MuzaffarWilliams}
H.~Muzaffar and K.~S. Williams.
\newblock {Evaluation of complete elliptic integrals of the first kind at
  singular moduli}.
\newblock {\em Taiwanese J. Math.}, 10(6):1633--1660, 2006.

\bibitem{RamanujanModularPi}
S.~Ramanujan.
\newblock {Modular equations and approximations to {$\pi$}}.
\newblock {\em Quart. J. Math.}, 45:350--372, 1914.

\bibitem{Ramanujan629}
S.~Ramanujan.
\newblock {Question 629}.
\newblock {\em J. Indian Math. Soc.}, 7:40, 1918.
\newblock {solutions to \emph{Question 629} by N.~Durai Rajan and M.~Bhimasena
  Rao in Volume 8, pp.~25--30}.

\bibitem{RamanujanCollected}
S.~Ramanujan.
\newblock {\em {Collected Papers of Srinivasa Ramanujan}}.
\newblock G.~H. Hardy, P.~V. Seshu Aiyar, and B.~M. Wilson, eds., Cambridge
  University Press, Cambridge, 1927.

\bibitem{RamanujanEarlierI}
S.~Ramanujan.
\newblock {\em {Notebooks of Srinivasa Ramanujan, Volume I.}}
\newblock Tata Institute of Fundamental Research, Bombay, 1957.

\bibitem{RamanujanEarlierII}
S.~Ramanujan.
\newblock {\em {Notebooks of Srinivasa Ramanujan, Volume II.}}
\newblock Tata Institute of Fundamental Research, Bombay, 1957.

\bibitem{RamanujanLost}
S.~Ramanujan.
\newblock {\em {The Lost Notebook and Other Unpublished Papers}}.
\newblock Narosa, New Delhi, 1988.

\bibitem{Rebak}
{\"{O}}.~Reb{\'{a}}k.
\newblock {The three missing terms in Ramanujan's septic theta function
  identity}.
\newblock {\em Ramanujan J.}, 60(4):885--911, 2023.

\bibitem{SelbergChowla}
A.~Selberg and S.~Chowla.
\newblock {On Epstein's zeta function}.
\newblock {\em J. Reine Angew. Math.}, 227:86--110, 1967.

\bibitem{Son}
S.~H. Son.
\newblock {Septic theta function identities in Ramanujan's lost notebook}.
\newblock {\em Acta Arith.}, 98(4):361--374, 2001.

\bibitem{Son2}
S.~H. Son.
\newblock {Ramanujan's symmetric theta functions in his lost notebook}.
\newblock In {\em {Special Functions and Orthogonal Polynomials}}, {D.~Dominici
  and R.~S.~Maier, eds., Contemp. Math., Vol. 471, pp.~187--202}. {American
  Mathematical Society}, {Providence, RI}, 2008.

\bibitem{Watson3}
G.~N. Watson.
\newblock {Singular moduli (3)}.
\newblock {\em Proc. London Math. Soc.}, 40(1):83--142, 1936.

\bibitem{Weber}
H.~Weber.
\newblock {\em {Lehrbuch der Algebra, Dritter Band}}.
\newblock {2nd ed., Druck und Verlag von Friedrich Vieweg und Sohn},
  Braunschweig, 1908.

\bibitem{WhittakerWatson}
E.~T. Whittaker and G.~N. Watson.
\newblock {\em {A Course of Modern Analysis}}.
\newblock 4th ed., Cambridge University Press, Cambridge, 1950.

\bibitem{Yi}
J.~Yi.
\newblock {Theta-function identities and the explicit formulas for
  theta-functions and their applications}.
\newblock {\em J. Math. Anal. Appl.}, 292(2):381--400, 2004.

\bibitem{Yi2}
J.~Yi, M.~G. Cho, J.~H. Kim, S.~H. Lee, J.~M. Yu, and D.~H. Paek.
\newblock {On some modular equations and their applications I}.
\newblock {\em Bull. Korean Math. Soc.}, 50(3):761--776, 2013.

\bibitem{Zucker}
I.~J. Zucker.
\newblock {The evaluation in terms of $\Gamma$-functions of the periods of
  elliptic curves admitting complex multiplication}.
\newblock {\em Math. Proc. Cambridge Philos. Soc.}, 82(1):111--118, 1977.

\end{thebibliography}
\end{document}